\documentclass{article}
\usepackage{amssymb, amsmath, latexsym, amsthm}
\usepackage[activeacute, activegrave, english]{babel}
\usepackage{hyperref}
\hypersetup{urlcolor=blue} 
\providecommand{\keywords}[1]{\textbf{Keywords:} #1}
\usepackage[sort&compress,square,comma,authoryear]{natbib}
\usepackage{longtable}
\usepackage{authblk}

\usepackage{tikz}
\usetikzlibrary{positioning,arrows,calc}
\tikzset{
modal/.style={>=stealth',shorten >=1pt,shorten <=1pt,auto,node distance=1.5cm,
semithick},
real world/.style={double,circle,draw,thick,align=center},
world/.style={circle,draw,minimum size=0.5cm,fill=gray!15},
point/.style={circle,draw,inner sep=0.5mm,fill=black},
reflexive above/.style={->,loop,looseness=7,in=120,out=60},
reflexive below/.style={->,loop,looseness=7,in=240,out=300},
reflexive left/.style={->,loop,looseness=7,in=150,out=210},
reflexive right/.style={->,loop,looseness=7,in=30,out=330}
}

\usepackage[toc]{appendix}

\newtheorem{defi}{Definition}
\newtheorem{prop}{Proposition}
\newtheorem{theor}{Theorem}
\newtheorem{lemma}{Lemma}

\date{}

\author{
  Ekaterina Kubyshkina\footnote{Department of Philosophy, University of Milan, \texttt{ekaterina.kubyshkina@unimi.it}}

  \and
  Mattia Petrolo\footnote{LASIGE Computer Science and Engineering Research Centre, University of Lisbon, \texttt{mpetrolo@fc.ul.pt}}

}

\title{Ignorance with(out) Grasping}

\begin{document}

\maketitle

\abstract

In this work, we argue that ignorance can be inherently understood as a hyperintensional notion. When faced with two logically or necessarily equivalent propositions, an agent may be ignorant of one while not of the other. To capture formally this intuition, we employ a topic-sensitive semantics, enabling the modeling of an agent's attitude toward the content of a proposition. Within this framework, we reevaluate three existing logical systems, usually characterized by standard Kripke semantics, to account for three forms of ignorance: ignorance whether, ignorance as unknown truth, and disbelieving ignorance. For each form, we present a sound and complete system. To highlight the advantage of this approach, we apply it to address the problem of logical omniscience rephrased in terms of ignorance. The resulting framework considers an agent's capacity to grasp the content of a proposition, bridging the gap between standard relational settings for ignorance representation and natural intuitions about the role of content in forming one's ignorance.

\

\section{Introduction}
\label{Intro}

Hyperintensional contexts, initially introduced by \citet{Cresswell1975} to distinguish between necessarily equivalent contents, have, in recent decades, garnered significant attention in both logical and philosophical investigations. Several notions have been identified as hyperintensional to some extent: information, belief, knowledge, meaning, explanation (see, e.g., \citet{Fine2017}, \citet{Leitgem2019}, \citet{Berto2019}, \citet{BertoH2021}). Furthermore, \citet{Wansing2017} argues that hoping, fearing, supposing, and imagining exhibit hyperintensional features. Building on this, \citet{BertoSEP} advance the thesis that: ``One may be led to the general conclusion that \textit{thought} is hyperintensional.'' A primary objective of the present work is to contribute new evidence to support this thesis by introducing a novel element to the plethora of hyperintensional notions: ignorance.

From an epistemological standpoint, ignorance is usually conceived as a multi-layered notion. In contrast to knowledge and belief, often regarded as paradigmatic cases of hyperintensional notions, ignorance is commonly understood as a notion that manifests in various kinds or forms. For instance, \citet{Peels2020} delineates six forms of (propositional) ignorance: disbelieving ignorance, suspending ignorance, undecided ignorance, unconsidered ignorance, deep ignorance, and complete ignorance.\footnote{We do not have to assume that these forms of ignorance are exhaustive. For our purposes here, it is sufficient to accept that ignorance is not a singular and monolithic notion.} Let us focus our attention on complete ignorance, although unconsidered ignorance and deep ignorance could serve the same purpose.  \citet{Peels2020} characterizes complete ignorance as ``$p$ is such that $S$ cannot grasp or entertain $p$,'' where $S$ is an epistemic subject and $p$ a true proposition. Due to this form of ignorance, it has been argued that ignorance, in its full generality, cannot be viewed as a mental state. This stems from the fact that complete ignorance lacks at least one of the typical features attributed to mental states, namely the aspect of having an attitude towards a proposition (see \citet{Kubyshkina2020}).

Furthermore, there exists a lack of consensus among epistemologists regarding the very definition of ignorance.  The Standard View defines ignorance merely as not knowing (see, e.g., \citet{LeMorvan2012}), while the New View defines it as the absence of true belief (see, e.g., \citet{Peels2012}). A crucial distinction between these positions lies in the acceptance (Standard View) or rejection (New View) of false propositions as genuine instances of ignorance. As a consequence, the New View, in contrast to the Standard View, contends that ignorance must be factive, meaning that if an agent is ignorant of $\phi$, then $\phi$ is true.

Finally, from a logical standpoint, ignorance currently lacks a standard formalization with well-established semantic properties and systems, comparable to the formalizations for knowledge and belief. Over the last few decades, several frameworks formalizing different aspects of ignorance have emerged. Detailed presentations of some of these frameworks will be provided in the next section. For the moment, it is sufficient to note that the multi-layered nature of ignorance and the diverse intuitions surrounding it contribute to a profound disagreement concerning the very foundations of this elusive notion.

The discussion on the epistemological and logical status of ignorance, as described, provides compelling evidence for the necessity of distinguishing this notion from other mental states usually analyzed in the literature on hyperintensional phenomena. Considering the substantial importance and ubiquity of ignorance, evident across scientific instances as well as in moral and social contexts, this constitutes one of the main motivations for  pursueing a dedicated inquiry about the hyperintensionality of this notion. 

The plan is as follows. Section \ref{ReprSec} introduces three formal frameworks for ignorance representation prevalent in the literature. Each of these frameworks introduces ignorance via a primitive modality, allowing one to model three kinds of ignorance: ignorance whether, ignorance as unknown truth, and disbelieving ignorance. In Section \ref{FormalIgnoranceSection}, we explore the rationale behind treating ignorance as a hyperintensional notion and present the semantics employed to model hyperintensional ignorance. The main results are sound and complete axiomatizations for each kind of ignorance. In Section \ref{SecOmniscience},  we tackle the problem of logical omniscience through the lens of ignorance, offering an application of our systems to address this problem. We conclude in Section \ref{Conclusion}.

\section{The formal representation of ignorance}
\label{ReprSec}

Recently, the formal representation of ignorance has attracted significant attention, not only in logic (see \cite{Steinsvold2008}, \cite{Fan2015}) and computer science (see \citet{Hoek2004}), but also in formal epistemology (see \citet{Kubyshkina2021}). In this section we provide an overview of three logics in which ignorance is represented via a primitive modality. Specifically, we will consider \textit{ignorance whether}, \textit{ignorance as unknown truth}, and the formal counterpart of \textit{disbelieving ignorance}. 

In this section, we utilize standard Kripke semantics, without modification to standard presentations. A \textit{frame} is a pair $(W, R)$ where $W$ is a non-empty set of possible worlds, and $R \subseteq W \times W$ is an accessibility relation. A \textit{model} is a pair $(\mathcal{F}, v)$ where $\mathcal{F}$ is frame and $v: \texttt{Prop} \rightarrow \mathcal{P}(W)$ is a valuation function which assigns to each propositional variable in \texttt{Prop} a set of possible worlds.

The first definition we consider is that of \citet{Hoek2004}, who take ignorance to be `not knowing whether' (that is, ignorance whether). The authors consider an extension of classical propositional logic with a primitive operator $I^{w}$ interpreted on Kripke models. The semantics is defined as follows.

\begin{defi}
\label{defSemIwhether}

Let $\mathcal{M} = (W, R, v)$, then

\begin{enumerate}

\item $\mathcal{M}, w \models p$ iff $w \in v(p)$;

\item $\mathcal{M}, w \models \neg \phi$ iff $\mathcal{M}, w \not \models \phi$;

\item $\mathcal{M}, w \models \phi \wedge \psi$ iff $\mathcal{M}, w \models \phi$ and $\mathcal{M}, w \models \psi$;

\item $\mathcal{M}, w \models I^{w} \phi$ iff there exist $w', w'' \in W$ such that $Rww'$, $Rww''$, $\mathcal{M}, w' \models \phi$, and $\mathcal{M}, w'' \models \neg \phi$.

\end{enumerate}

\end{defi}

The propositional operators ($\vee$, $\rightarrow$, $\leftrightarrow$) are defined standardly. Expressions of the form `$I^{w}\phi$' are to be read as `the agent is ignorant whether $\phi$.' 

\citet{Fan2015} show that this semantics characterizes the following system.\footnote{Technically, the authors provide the result for the noncontingency operator $\triangle$. As noted in \citep[p. 1295]{Fan2021}, ``Just as knowledge is the epistemic counterpart of necessity, the forms (iii) [ignorance as unknown truth] and (iv) [ignorance whether] are the epistemic counterparts of accident and contingency.'' In this subsection, we address ignorance whether, and we will discuss ignorance as unknown truth in the subsequent subsection. To uniform the notation, we use the operator $I^{w}$ for ignorance whether and $I^{u}$ for ignorance as unknown truth in our presentation of the systems.} We dub this system \texttt{IW}, which stands for `Ignorance Whether.'

\begin{defi}[System \texttt{IW}]

\

\begin{itemize}

\item[(CPL)] All classical propositional tautologies and Modus Ponens

\item[(A1$_{IW}$)] $(\neg I^{w} (\phi \rightarrow \psi) \wedge \neg I^{w} (\neg \phi \rightarrow \psi)) \rightarrow \neg I^{w} \psi$

\item[(A2$_{IW}$)] $\neg I^{w} \phi \rightarrow (\neg I^{w}(\phi \rightarrow \psi) \vee \neg I^{w}(\neg \phi \rightarrow \chi)$

\item[(A3$_{IW}$)] $\neg I^{w} \phi \leftrightarrow \neg I^{w} \neg \phi$

\item[(NEC$_{IW}$)] From $\vdash \phi$ infer $\vdash \neg I^{w} \phi$

\item[(RE$_{IW}$)] From $\vdash \phi \leftrightarrow \psi$ infer $\vdash \neg I^{w} \phi \leftrightarrow \neg I^{w} \psi$

\end{itemize}

\end{defi}

Notice that a formula $I^{w} \phi$ can be defined as $\neg \Box \phi \wedge \neg \Box \neg \phi$, where $\Box$ is defined in the standard way. However, as shown in  \citep{Fan2015}, the $\Box$ operator cannot be defined using $I^{w}$ on \texttt{K}-frames (i.e., frames where the accessibility relation is not restricted by any property).

\

The second operator for ignorance is introduced by \citet{Steinsvold2008}. In this case, ignorance of $\phi$ is considered as if $\phi$ is true, but unknown to the agent. This kind of ignorance is represented semantically using a primitive modality denoted by $I^{u}$. On Kripke models, the operator is defined as follows.

\begin{defi} Let $\mathcal{M} = (W, R, v)$.

\begin{itemize}

\item $\mathcal{M}, w \models I^{u} \phi$ iff $\mathcal{M}, w \models \phi$ and there exists $w'$ such that $Rww'$ and $\mathcal{M}, w' \models \neg \phi$.

\end{itemize}

\end{defi}

By replacing the clause 4 of Definition \ref{defSemIwhether} with the clause above, we get the semantics which, as shown by \cite{Fan2015a}), characterizes the system \texttt{IU}, standing for `Ignorance as Unknown truth.'

\begin{defi}[System \texttt{IU}]

\

\begin{itemize}

\item[(CPL)] All classical propositional tautologies and Modus Ponens

\item[(US)] Uniform Substitution

\item[(A1$_{IU}$)] $\neg I^{u} \top$

\item[(A2$_{IU}$)] $\neg p \rightarrow \neg I^{u} p$

\item[(A3$_{IU}$)] $(\neg I^{u} p \wedge \neg I^{u} q) \rightarrow \neg I^{u}(p \wedge q)$

\item[(R$_{IU}$)] From $\vdash \phi \rightarrow \psi$ infer $\vdash (\neg I^{u} \phi \wedge \phi) \rightarrow \neg I^{u} \psi$

\end{itemize}

\end{defi}

Similarly to $I^{w}$, formula $I^{u} \phi$ can be defined as $\phi \wedge \neg \Box \phi$, but $\Box$ cannot be defined using the $\Box$ operator on \texttt{K}-frames (see \citet{Gilbert2016}).

\

A third option for ignorance representation is provided by \citet{Kubyshkina2021}. We term this type of ignorance \textit{disbelieving ignorance}, and denote $I^{d}$ the primitive modality associated with it.\footnote{\citet{Kubyshkina2021} refer to the type of ignorance represented by $I^{d}$ as `factive ignorance.' However, $I^{d}$ is not the only factive operator for ignorance; for example, $I^{u}$ is also factive. Therefore, we identify $I^{d}$ with a specific type of ignorance--disbelieving ignorance--by borrowing the terminology from \citep{Peels2020}.} An agent is disbelievingly ignorant of $\phi$ if $\phi$ is true, but the agent believes $\neg \phi$. For example, an agent believes that her brother is in Rome, but he is not. Thus, she is dibelievingly ignorant that her brother is not in Rome. The definition of $I^{d}$ on standard Kripke models follows.

\begin{defi}[Operator $I^{d}$]

Let $\mathcal{M} = (W, R, v)$.

\begin{itemize}

\item  $\mathcal{M}, w \models I^{d} \phi$ iff $\mathcal{M}, w \models \phi$ and for all $w'$ which are not $w$, if $Rww'$ then $\mathcal{M}, w'  \models \neg \phi$.

\end{itemize}

\end{defi}

The semantic clause for $I^{d} \phi$ to be true in a world $w$ states that $\phi$ is true in $w$, but considered false in all worlds accessible from $w$ (which are not $w$ itself). This characterization justifies the name `disbelieving ignorance': for everything that the agent considers to be true, $\phi$ is false.\footnote{\citet{Steinsvold2011} introduced an operator for ``being wrong,'' denoted by $W$, which can also be interpreted as an operator for disbelieving ignorance. In \citet{Gilbert2021}, it was shown that $W$ and $I^{d}$ give rise to the same set of validities. Consequently, the results presented in this article can also be extended to the system containing the $W$ operator.}

By replacing the clause 4 of Definition \ref{defSemIwhether} with the clause above, we get the semantics which, as shown by \citet{Gilbert2021}, characterizes the system \texttt{DI}, standing for `Disbelieving Ignorance.'

\begin{defi}[System \texttt{DI}]

\

\begin{itemize}

\item[(CPL)] All classical propositional tautologies and Modus Ponens

\item[(US)] Uniform Substitution

\item[(A1$_{DI}$)] $I^{d} p \rightarrow p$

\item[(A2$_{DI}$)] $(I^{d} p \wedge I^{d} q) \rightarrow I^{d}(p \vee q)$

\item[(IR$_{DI}$)] From $\vdash \phi \rightarrow \psi$ infer $\vdash \phi \rightarrow (I^{d} \psi \rightarrow I^{d} \phi)$ 

\end{itemize}

\end{defi}

As shown by \citet{Gilbert2021}, $I^{d}$ and $\Box$ are not interdefinable on \texttt{K}-frames (as well as on any reflexive frame). 

\

Evidently, none of the systems \texttt{IW}, \texttt{IU}, \texttt{DI} treats ignorance as a hyperintensional notion. In each case, if $\phi \leftrightarrow \psi$, then $I \phi \leftrightarrow I \psi$, where $I$ stands for $I^{w}$, $I^{u}$, or $I^{d}$.  

\section{The hyperintensionality of ignorance}
\label{FormalIgnoranceSection}

What justifies viewing ignorance as a hyperintensional notion to begin with? The answer to this question depends on the perspective one adopts regarding the nature of ignorance. In Section \ref{Intro}, we characterized the Standard View as identifying ignorance with not knowing. If one accepts this view, then the hyperintensionality of ignorance follows from the hyperintensionality of knowledge. Given any hyperintensional case of knowledge -- where one knows $\phi$ but not $\psi$, despite $\phi$ and $\psi$ being logically equivalent -- one is ignorant of $\psi$ but not $\phi$. However, this becomes less clear if one does not accept the Standard View. The argument hinges on the complementarity between knowledge and ignorance, which can be questioned. Consider a situation where an agent does not know that $1 + 1 = 3$ because she knows that $1 + 1 \not = 3$. If one assumes the Standard View, then one must conclude that the agent is ignorant of $1 + 1 = 3$. However, several authors argue that this is not a case of ignorance (see, e.g., \citet{vanWoudenberg2009}, \citet{Zimmerman2018}, \citet{Pritchard2021}, \citet{Peels2023}). The reason may lie either in the assumption that ignorance is factive, or in the idea that knowledge of $\phi$ implies non-ignorance of $\neg \phi$. In both cases, the example highlights that even though ignorance implies non-knowledge, it is not always the case that non-knowledge implies ignorance. Therefore, if the complementarity between knowledge and ignorance is rejected, it is no longer guaranteed that hyperintensionality of ignorance follows from the hyperintensionality of non-knowledge. 

Moreover, from a logical perspective, the three systems presented in Section \ref{ReprSec} do not consider ignorance merely as not knowing, but rather identify more restricted forms of ignorance by specifying the reasons why an agent is ignorant. Thus, none of the operators $I^{w}$, $I^{u}$, or $I^{d}$ aligns perfectly with the Standard View: $I^{w}$ and $I^{u}$ are defined as conjunctions of not knowing with additional conditions, while $I^{d}$ is not definable in terms of knowledge. This does not demonstrate an absolute incompatibility between the systems \texttt{IW}, \texttt{IU}, \texttt{DI} and the Standard View, but it would be incorrect to assume that the Standard View is necessarily the underlying framework for the three systems. 

Finally, none of the systems \texttt{IW}, \texttt{IU}, or \texttt{DI} allows for the representation of  knowledge understood as the epistemic counterpart of the $\Box$ operator. Therfore, since knowledge is inexpressible in these settings, its hyperintensional nature cannot be straightforwardly applied to ignorance. 

To justify that ignorance is a hyperintensional notion, independently of the Standard View, let us borrow and adapt a famous example from \citet[p. 88]{Stalnaker1984}.\footnote{We thank an anonymous reviewer for bringing Stalnaker's example to our attention.} William III of England was not ignorant that England could avoid a war with France, but he was ignorant of the concept of nuclear war. Therefore, it would be odd to claim that he was not ignorant of whether England could avoid war with France using nuclear weapons. Let $p$ stand for ``England can avoid war with France'' and $q$ stand for ``England can avoid nuclear war with France.''  Thus, if William III is ignorant whether England could avoid war with France using nuclear  weapons ($q \vee \neg q$),\footnote{In this case, William III is ignorant of a tautology. This does not imply that he is unable to determine whether $q \vee \neg q$ holds when presented with the formula, but rather that he is not in a position to initiate the deductive inquiry leading to this conclusion. According to Stalnaker's analysis, William III lacks the information necessary to reason about nuclear wars.} he is also ignorant of $p \wedge (q \vee \neg q)$. However, this means that he should also be ignorant of whether England could avoid a war with France, because $p$ is logically equivalent to $p \wedge (q \vee \neg q)$. Therefore, we have an example where an agent is ignorant of one proposition but not another, despite the fact that they have the same intension. To conclude, we have shown that regardless of whether one accepts or rejects the Standard View, the conclusion holds true: ignorance should be understood as a hyperintensional notion.

The following subsections present three hyperintesional systems for representing ignorance whether, ignorance as unknown truths, and disbelieving ignorance. In particular, we redefine the operators $I^{w}$, $I^{u}$, and $I^{d}$ by adding a supplementary condition on grasping or not the content of the proposition. This permits us to distinguish between situations in which an agent is ignorant of $\phi$ and not ignorant of $\psi$, even though $\phi$ and $\psi$ are equivalent, on the basis of what the agent grasps.

\subsection{A logic for Hyperintensional Ignorance Whether}
\label{HIWsubsection}

Let us start by examining $I^{w}$. As it was pointed out previously, this kind of ignorance manifests when an agent possesses no knowledge regarding the veracity of a proposition or its negation. The primary objective of this subsection is to refine this definition by incorporating the relevance of the topic of a proposition about which an agent may be ignorant. In particular, our objective is to elucidate the additional condition for ignorance: an agent may be ignorant whether because she is not grasping the topic of a given proposition. 

We conceive topics in the same way as \cite{Ozgun2020}, who, in turn, follow Yablo's idea of identifying topics with what `meaningful items bear to whatever it is that they are \textit{on} or \textit{of} or that they \textit{address} or \textit{concern}'' (\cite{Yablo2014}, p. 1). A topic of a proposition is what the proposition is about. Therefore, the semantics to model this idea must account for two components: the intension of a proposition $\phi$ (i.e., the set of possible worlds), and its topic.\footnote{Topics have been treated differently by various authors. For more details, see \citep{Ozgun2020} and the references therein.}

Formally, we proceed by defining recursively the language $\mathcal{L}_{I^{w}}$ as follows:

$$ \phi : = p \mid \neg \phi \mid \phi \wedge \phi \mid I^{w} \phi \mid \Box \phi$$

The propositional operators ($\vee$, $\rightarrow$, $\leftrightarrow$) are defined standardly. Expressions of the form `$I^{w}\phi$' are to be read as `the agent is ignorant whether $\phi$.' In lign with the framework provided by \citet{Ozgun2023}, we incorporate the $\Box$ operator into the language for technical reasons. Conceptually, it can be understood as an \textit{a priori} modality, signifying analytic truths. It is noteworthy that both modalities, $I^{w}$ and $\Box$, are primitive in this language.

In order to provide the semantics for the language $\mathcal{L}_{I^{w}}$, we use topic and topic-sensitive models defined in \citep[Definitions 1,2]{Ozgun2023}.

\begin{defi}[Topic model]
\label{topicmodel}

A topic is a tuple $\mathcal{T} = (T, \oplus, t, \mathfrak{K})$ where

\begin{itemize}

\item $T$ is a non-empty set of possible topics;

\item $\oplus: T \times T \mapsto T$ is an idempotent, commutative, associative topic-fusion operator;

\item $\mathfrak{K} \in T$ is a designated topic representing the totality of topics grasped by the agent; and

\item $t: \mathtt{Prop} \mapsto T$ is a topic function assigning a topic to each element in $\mathtt{Prop}$.

\end{itemize}

\end{defi}

\begin{defi}[Topic-sensitive model]

A topic-sensitive model is a tuple $\mathcal{M} = (W, R, v, \mathcal{T})$ where $W$ is a non-empty set of possible worlds, $R \subseteq W \times W$ is a binary accessibility relation between worlds, $v: \mathtt{Prop} \mapsto \mathcal{P}(W)$ is a standard valuation function that assigns to each propositional variable a set of possible worlds, and $\mathcal{T}$ is a topic model as given in Def. \ref{topicmodel}.

\end{defi}

The function $t$ extends to the language $\mathcal{L}_{I^{w}}$ by taking the topic of $\phi$ as the fusion of the topics of the elements in $Var(\phi): t(\phi) = \oplus \{t(p): p \in Var(\phi)\}$, where $Var(\phi)$ denotes the set of propositional variables occurring in $\phi$. This entails topic-transparency of operators: $t(\phi) = t(I^{w}\phi) = t(\Box \phi) = t(\neg \phi)$ and $t(\phi \wedge \psi) = t(\phi) \oplus t(\psi)$. Topic parthood $\sqsubseteq$ is defined as follows: for all $a, b \in T$, $a \sqsubseteq b$ iff $a \oplus b = b$. It follows that $(T, \oplus)$ is a join semilattice and $(T, \sqsubseteq)$ is a partially ordered set.

\begin{defi}[Semantics for $\mathcal{L}_{I^{w}}$]
\label{semHIW}

Let $\mathcal{M} = (W, R, v, \mathcal{T})$, then

\begin{enumerate}

\item $\mathcal{M}, w \models p$ iff $w \in v(p)$;

\item $\mathcal{M}, w \models \neg \phi$ iff $\mathcal{M}, w \not \models \phi$;

\item $\mathcal{M}, w \models \phi \wedge \psi$ iff $\mathcal{M}, w \models \phi$ and $\mathcal{M}, w \models \psi$;

\item $\mathcal{M}, w \models \Box \phi$ iff for all $w' \in W$, $\mathcal{M}, w' \models \phi$;

\item $\mathcal{M}, w \models I^{w} \phi$ iff $t(\phi) \not \sqsubseteq \mathfrak{K}$ or (there exist $w', w'' \in W$ such that $Rww'$, $Rww''$, $\mathcal{M}, w' \models \phi$, and $\mathcal{M}, w'' \models \neg \phi$).

\end{enumerate}

A formula $\phi$ is valid  in $\mathcal{M}$ ($\mathcal{M} \models \phi$), if $\mathcal{M}, w \models \phi$ for all $w \in W$. We call a formula $\phi$ valid in a class of topic-sensitive models $\mathfrak{C}$ if $\mathcal{M} \models \phi$ for all $\mathcal{M} \in \mathfrak{C}$.

\end{defi}

As it is clear from the case 4 of Definition \ref{semHIW}, $\Box$ is the global modality which incapsulates $S5$ properties, while the interpretation of $I^{w}$ does not require any supplementary property of the accessibility relation. 

Let $Var(\phi)$ denote the set of propositional variables occurring in $\phi$, $\bar{\phi} : = \bigwedge_{x \in Var(\phi)} (x \vee \neg x)$. Similarly to \citep{Berto2021} and \citep{Ozgun2023}, we need this kind of expressions for simplifying the axiomatization of the system. Given a model $\mathcal{M} = (W, R, v, \mathcal{T})$, for $\neg I^{w} \phi$ to be true at $w$ we should have $t(\phi) \sqsubseteq \mathfrak{K}$ and for all worlds $w' \in W$ s.t. $Rww'$ and either $\phi$ is true in all $w'$, or $\neg \phi$ is true in all $w'$. Since $\bar{\phi}$ is true everywhere, $Var(\bar{\phi}) = Var(\phi)$ for any $\phi \in \mathcal{L}_{I^{w}}$, and $\mathcal{M}, w \models \neg \bar{\phi}$ is impossible for any $w \in W$ for any $\mathcal{M}$, formulas of the form $\neg I^{w} \bar{\phi}$ ($I^{w} \bar{\phi}$) express statements such as `an agent has (not) grasped the topic of $\phi$,' respectively. The proof can be summarized as follows:

\begin{center}

$\mathcal{M}, w \models \neg I^{w} \bar{\phi}$ iff $t(\phi) \sqsubseteq \mathfrak{K}$ and 

($\mathcal{M}, w' \models \bar{\phi}$ for all $w'$ s.t. $Rww'$ or $\mathcal{M}, w' \models \neg \bar{\phi}$ for all $w'$ s.t. $Rww'$) iff

$t(\phi) \sqsubseteq \mathfrak{K}$ and $\mathcal{M}, w' \models \bar{\phi}$ for all $w'$ s.t. $Rww'$, iff

$t(\phi) \sqsubseteq \mathfrak{K}$.

\end{center}

The reading of $\neg I^{w} \bar{\phi}$ and of $I^{w} \bar{\phi}$ in terms of grasping helps in understanding the axioms of the system $\mathtt{HIW}$, which stands for `Hyperintensional Ignorance Whether.'

\begin{defi}[System $\mathtt{HIW}$]
\label{HIW}

\

\begin{itemize}

\item[(CPL)] All classical propositional tautologies and Modus Ponens

\item[(S5$_{\Box}$)] All S5 axioms and rules for $\Box$

\item[($I^{w} \leftrightarrow$)] $I^{w} \phi \leftrightarrow I^{w} \neg \phi$

\item[(A1$_{I^{w}}$)] $(\neg I^{w}\phi \wedge \neg I^{w} \bar{\psi} \wedge \neg I^{w} \bar{\chi}) \rightarrow (\neg I^{w}(\phi \rightarrow \psi) \vee \neg I^{w}(\neg \phi \rightarrow \chi))$

\item[(A2$_{I^{w}}$)] $\neg I^{w} \bar{\phi} \rightarrow \neg I^{w}( \phi \vee \neg \phi)$

\item[(A3$_{I^{w}}$)] $(\neg I^{w} \phi \wedge \neg I^{w} \psi) \rightarrow \neg I^{w} (\phi \wedge \psi)$

\item[(A4$_{I^{w}}$)] $\neg I^{w} \bar{\phi} \rightarrow \Box \neg I^{w} \bar{\phi}$

\item[(A5$_{I^{w}}$)] $I^{w} \bar{\phi} \rightarrow I^{w} \phi$

\item[(A6$_{I^{w}}$)] $(\neg I^{w} (\phi \rightarrow \neg \psi) \wedge \neg I^{w} \phi \wedge \neg I^{w} (\psi \rightarrow \phi)) \rightarrow \neg I^{w} \psi$

\item[(A7$_{I^{w}}$)] $\neg I^{w} \bar{\phi} \rightarrow \neg I^{w} \bar{\psi}$, if $Var(\psi) \subseteq Var(\phi)$

\item[(NEC$_{I^{w}}$)] From $\vdash \phi$ infer $\vdash \neg I^{w} \bar{\phi} \rightarrow \neg I^{w} \phi$

\item[(RE$_{I^{w}}$)] From $\vdash \phi \leftrightarrow \psi$ infer $\vdash (\neg I^{w}\bar{\phi} \wedge \neg I^{w} \bar{\psi}) \rightarrow (\neg I^{w} \phi \leftrightarrow \neg I^{w} \psi)$

\end{itemize}

\end{defi}

The axiom ($I^{w} \leftrightarrow$) is a standard axiom for $I^{w}$ (see axiom (A3$_{IW}$) of the system \texttt{IW}), claiming that being ignorant whether $\phi$ is equivalent to being ignorant whether $\neg \phi$. The axiom (A1$_{I^{w}}$) is a refinement of the axiom (A2$_{IW}$)  of the system \texttt{IW}, claiming that if an agent is not ignorant whether $\phi$, then she is not ignorant whether either of the consequences of $\phi$, or of $\neg \phi$, given that the topics of these ocnsequences are grasped. The axiom (A2$_{I^{w}}$) means that if an agent grasps the topic of $\phi$, then she is not ignorant whether the excluded middle for $\phi$ holds. The axiom (A3$_{I^{w}}$) means that non-ignorance of $\phi$ and $\psi$ implies non-ignorance of their conjunction. The axiom (A4$_{I^{w}}$) means that grasping the topic of $\phi$ implies that the agent grasps it in all epistemic states. The axiom (A5$_{I^{w}}$) means that if an agent does not grasp the topic of a proposition, then she is ignorant of this proposition. The axiom (A6$_{I^{w}}$) means that if an agent is not ignorant of a proposition $\phi$ and not ignorant on the relation between $\phi$ and some proposition $\psi$, then the agent is not ignorant whether $\psi$ holds. The axiom (A7$_{I^{w}}$) means that if an agent grasps the content of $\psi$, then she grasps the content of all its subformulas. The rules (NEC$_{I^{w}}$)  and (RE$_{I^{w}}$) are refinements of standard rules NEC and RE for $\neg I^{w}$ interpreted on topic-sensitive semantics. The (NEC$_{I^{w}}$) claims that whenever there is a theorem, the agent is in not ignorant whether it is true only if she grasps the topic of this theorem. The (RE$_{I^{w}}$) rule claims that whenever there are two  equivalent propositions, the agent is in the same epistemic state with respect to both propositions, only if she grasps the topics of these propositions. As will become clear in Section \ref{SecOmniscience}, these refinements explicitly highlight an important feature of hyperintensional ignorance, which proves essential in addressing the problem of logical omniscience.

\begin{theor}
\label{theorHIWsound}

The system $\mathtt{HIW}$ is sound.

\end{theor}

\begin{theor}
\label{theorHIWcomplete}

The system $\mathtt{HIW}$ is complete.

\end{theor}

The proof is in Appendices \ref{appHIWsound} and \ref{appHIWcomplete}.

\

Let us now reconsider the example of William III that we have provided at the beginning of this section in terms of hyperintensional ignorance whether. We need to construct a model that distinguishes between the following 2 propositions:

(1$^{HIW}$) Willian III is ignorant whether $p$.

(2$^{HIW}$) William III is ignorant whether $p \wedge (q \vee \neg q)$.

Let $\mathcal{M} = (W, R, v, \mathcal{T})$, where $W = \{w, w'\}$, $R = \{(w,w), (w,w')\}$, $v(p) = v(q) = \{w, w'\}$, and $\mathcal{T} = (T, \oplus, t, \mathfrak{K})$, where $T =\{a, b\}$, $b \sqsubset a$, $t(p) = b$, $t(q) = a$, and $b = \mathfrak{K}$. In this model, we have $\mathcal{M} \models p \leftrightarrow (p \wedge (q \vee \neg q))$. We have that $\mathcal{M}, w \models \neg I^{w} p$, because for all $w'$ such that if $Rww'$ then $\mathcal{M}, w' \models p$ and $t(p) \sqsubseteq \mathfrak{K}$, and $\mathcal{M}, w \models I^{w} (p \wedge (q \vee \neg q))$, because $\mathfrak{K} \sqsubset t(q) = a$, that is $t(q) \not \sqsubseteq \mathfrak{K}$. Hence, despite the equivalence of $p$ and $p \wedge (q \vee \neg q)$, while William III is not ignorant whether $p$, he is ignorant whether $p \wedge (q \vee \neg q)$ because he does not grasp the topic of $q$.

\subsection{A logic for Hyperintensional Ignorance as Unknown truths}

Let us now consider $I^{u}$. Similarly to the case of $I^{w}$, our objective is to supplement the definition of this operator with topicality. In what follows, we will consider one possible refinement: we consider that an agent is ignorant of a proposition either if the proposition is true, but unknown by the agent, or because the agent does not grasp the content of this proposition. 

The language $\mathcal{L}_{I^{u}}$ is defined recursively as follows:

$$ \phi : = p \mid \neg \phi \mid \phi \wedge \phi \mid I^{u} \phi \mid \Box \phi$$

Other propositional operators are defined standardly. Expressions of the form `$I^{u}\phi$' should be read as `the agent is ignorant that $\phi$.'

The semantics is defined on the topic-sensitive models. We simply replace the case of $I^{w}$ by the following definition for $I^{u}$.

\begin{defi}[Semantics for $\mathcal{L}_{I^{u}}$]
\label{HIUdefi}

Let $\mathcal{M} = (W, R, v, \mathcal{T})$, then

\

\begin{itemize}

\item $\mathcal{M}, w \models I^{u} \phi$ iff $t(\phi) \not \sqsubseteq \mathfrak{K}$ or ($\mathcal{M}, w \models \phi$ and there exist $w' \in W$ such that $Rww'$ and $\mathcal{M}, w' \models \neg \phi$.

\end{itemize}

\end{defi}

Similarly to $I^{w}$, the expressions of the form $\neg I^{u} \bar{\phi}$ ($I^{u} \bar{\phi}$) should be read as `an agent has (not) grasped the topic of $\phi$':

\begin{center}

$\mathcal{M}, w \models \neg I^{u} \bar{\phi}$ iff $t(\phi) \sqsubseteq \mathfrak{K}$ and

$(\mathcal{M}, w \not \models \bar{\phi}$ or $\mathcal{M}, w' \models \bar{\phi}$ for all $w'$, s.t. $Rww'$) iff

$t(\phi) \sqsubseteq \mathfrak{K}$ and $\mathcal{M}, w' \models \bar{\phi}$ for all $w'$, s.t. $Rww'$, iff

$t(\phi) \sqsubseteq \mathfrak{K}$.

\end{center}

We now define the system $\mathtt{HIU}$, which stands for `Hyperintensional Ignorance as Unknown truths.'

\begin{defi}[System $\mathtt{HIU}$]

\

\begin{itemize}

\item[(CPL)] All classical propositional tautologies and Modus Ponens

\item[(S5$_{\Box}$)] All $S5$ axioms and rules for $\Box$

\item[(A1$_{I^{u}}$)] $(I^{u} \phi \wedge \neg I^{u} \bar{\phi}) \rightarrow \phi$

\item[(A2$_{I^{u}}$)] $\neg I^{u} \bar{\phi} \rightarrow \neg I^{u}( \phi \vee \neg \phi)$

\item[(A3$_{I^{u}}$)] $(\neg I^{u} \phi \wedge \neg I^{u} \psi) \rightarrow \neg I^{u}(\phi \wedge \psi)$

\item[(A4$_{I^{u}}$)] $\neg I^{u}\bar{\phi} \rightarrow \Box \neg I^{u} \bar{\phi}$

\item[(A5$_{I^{u}}$)] $I^{u} \bar{\phi} \rightarrow I^{u} \phi$

\item[(A6$_{I^{u}}$)] $(\phi \wedge \neg I^{u} \bar{\phi} \wedge \neg I^{u} \bar{\psi} \wedge I^{u}(\phi \vee \psi)) \rightarrow I^{u} \phi$

\item[(A7$_{I^{u}}$)] $(\neg I^{u} (\phi \vee \psi) \wedge \neg I^{u}( \phi \rightarrow \psi)) \rightarrow \neg I^{u} \psi$

\item[(A8$_{I^{u}}$)] $\neg I^{u} \bar{\phi} \rightarrow \neg I^{u} \bar{\psi}$, if $Var(\psi) \subseteq Var(\phi)$

\item[(NEC$_{I^{u}}$)] From $\vdash \phi$ infer $\vdash \neg I^{u} \bar{\phi} \rightarrow \neg I^{u} \phi$

\item[(RE$_{I^{u}}$)] From $\vdash \phi \leftrightarrow \psi$ infer $\vdash (\neg I^{u} \bar{\phi} \wedge \neg I^{u} \bar{\psi}) \rightarrow (\neg I^{u} \phi \leftrightarrow \neg I^{u} \psi)$

\end{itemize}

\end{defi}

The axiom (A1$_{I^{u}}$) redefines the relation between ignorance and factivity. Differently to \texttt{IU}, the operator $I^{u}$ in \texttt{HIU} is not factive: it is not the case that ignorance of $\phi$ implies the truth of $\phi$.\footnote{It is possible to redefine $I^{u}$ in a way that it remains factive. However, we leave this task for future investigations.} However,  (A1$_{I^{u}}$) preserves a form of factivity, since it claims that if one is ignorant of $\phi$ and grasps the topic of $\phi$, then $\phi$ is true. The reading of the axioms (A2$_{I^{w}}$) - (A5$_{I^{u}}$) is similar to the reading of the axioms (A2$_{I^{w}}$) - (A5$_{I^{w}}$), respectively. The axiom (A6$_{I^{u}}$) states that if $\phi$ is true, the agent is ignorant that $\phi \vee \psi$ for some $\psi$, but she grasps the topics of both $\phi$ and $\psi$, then she is ignorant that $\phi$. The axiom (A7$_{I^{u}}$) states that if an agent is not ignorant that $\phi \vee \psi$ and that $\phi \rightarrow \psi$, then she is not ignorant of $\psi$. The reading of the axiom (A8$_{I^{u}}$) is similar to the reading of the axiom (A7$_{I^{w}}$). The rule (NEC$_{I^{u}}$)  and (RE$_{I^{u}}$) are similar to (NEC$_{I^{w}}$) and (RE$_{I^{w}}$) respectively, and they will be useful in dealing with logical omniscience in the following section.

\begin{theor}
\label{theorHIUsound}

The system $\mathtt{HIU}$ is sound.

\end{theor}

\begin{theor}
\label{theorHIUcomplete}

The system $\mathtt{HIU}$ is complete.

\end{theor}

The proofs can be found in Appendices \ref{appHIUsound} and \ref{appHIUcomplete}.

\

It is easy to see that the example of William III presented at the beginning of this section can be handled precisely as in \texttt{HIW} by constructing a model in which the two logically equivalent propositions, $p$ and $p \wedge (q \vee \neg q)$, are distinguishable based on the agent's state of ignorance.

\subsection{A logic for Hyperintensional Disbelieving Ignorance}

Now, let us turn our attention to $I^{d}$. This operator, differently from $I^{w}$ and $I^{u}$, is not  reducible to standard $\Box$. Thus, it represents a distinct kind of formalism, thereby introducing a different challenge when it comes to incorporating hyperintensionality.

The language $\mathcal{L}_{I^{d}, G}$ is defined as follows:

\begin{center}

$ \phi:= p \mid \neg \phi \mid \phi \wedge \phi \mid I^{d} \phi \mid G \phi \mid \Box \phi$

\end{center}

Propositional operators are defined standardly. Expressions of the form `$I^{d} \phi$' should be read as `the agent is disbelievingly ignorant of $\phi$.' Expressions of the form `$G \phi$' should be read as `the agent grasps the topic of $\phi$.' In contrast to \texttt{HIW} and \texttt{HIU}, the operator $G$ is introduced as a primitive one. This choice is made because the semantic conditions for the $I^{d}$ operator do not readily lend themselves to a straightforward definition of grasping.

The semantics is defined on the topic-sensitive models. We replace the case of $I^{w}$ (or $I^{u}$) with $I^{d}$ and add the operator $G$ as follows.

\begin{defi}[Semantics for $\mathcal{L}_{I^{d}, G}$]

Let $\mathcal{M} = (W, R, v, \mathcal{T})$, then

\begin{itemize}

\item $\mathcal{M}, w \models I^{d} \phi$ iff $t(\phi) \sqsubseteq \mathfrak{K}$ and $\mathcal{M}, w \models \phi$ and for all $w'$ which are not $w$, if $Rww'$, then $\mathcal{M}, w' \models \neg \phi$;

\item $\mathcal{M}, w \models G \phi$ iff $t(\phi) \sqsubseteq \mathfrak{K}$.

\end{itemize}

\end{defi}

Note that the definition of $I^{d}$ differs from the definitions of $I^{w}$ and $I^{u}$. In the cases of $I^{w}$ and $I^{u}$, the absence of an agent's grasp of a proposition's topic constitutes a sufficient condition to deem the agent ignorant of the proposition. However, in the case of disbelieving ignorance, such a consideration seems unnatural. In particular, we consider that, to be disbelievingly ignorant of $\phi$, an agent should grasp the topic of $\phi$. For instance, an agent can be disbelievingly ignorant of the fact that ``Rome is the capital of Italy'' only if she is (falsely) convinced that this is not the case. Such a conviction presupposes a grasping of the topic of the proposition. Otherwise, if the agent does not grasp the topic of ``Rome is the capital of Italy,'' she would instead experience a different type of ignorance--namely, complete ignorance.

Now we define system $\mathtt{HDI}$, which stands for `Hyperintensional Disbelieving Ignorance.'

\begin{defi}[System $\mathtt{HDI}$]

\

\begin{itemize}

\item[(CPL)] All classical propositional tautologies and Modus Ponens

\item[(S5$_{\Box}$)] All S5 axioms and rules for $\Box$

\item[(G1)] $G \phi \leftrightarrow G\neg \phi$

\item[(G2)] $G \phi \leftrightarrow G I^{d} \phi$

\item[(G3)] $G\phi  \leftrightarrow G \Box \phi$

\item[(G4)] $(G \phi \wedge G \psi) \leftrightarrow G (\phi \wedge \psi)$

\item[(G5)] $G\phi \rightarrow \Box G \phi$

\item[(HIR)] From $\vdash \phi \rightarrow \psi$, infer $\vdash \phi \rightarrow (I^{d} \psi \rightarrow (G \phi \rightarrow I^{d} \phi))$

\item[(A1$_{I^{d}}$)] $I^{d} \phi \rightarrow \phi$

\item[(A2$_{I^{d}}$)] $(I^{d} \phi \wedge I^{d} \psi) \rightarrow I^{d}(\phi \vee \psi)$

\item[(A3$_{I^{d}}$)] $I^{d} \phi \rightarrow G \phi$

\end{itemize}

\end{defi}

The axioms (G1) - (G4) define the grasping operator as topic-transparent. The reading of the axiom (G5) is the same as of the axiom (A4$_{I^{w}}$) or (A4$_{I^{u}}$). The rule (HIR) means that whenever $\psi$ is a logical consequence of $\phi$, the truth of $\phi$, disbelieving ignorance of $\psi$ and grasping of $\phi$ imply disbelieving ignorance of $\phi$. The axiom (A1$_{I^{d}}$) expresses the factivity of disbelieving ignorance. The axiom (A2$_{I^{d}}$) states that if one is disbelievingly ignorant of $\phi$ and of $\psi$, then she is disbelievingly ignorant of their disjunction. 

\begin{theor}

The system $\mathtt{HDI}$ is sound.

\end{theor}

\begin{theor}

The system $\mathtt{HDI}$ is complete.

\end{theor}

The proofs can be found in Appendices \ref{appHDIsound} and \ref{appHDIcomplete}.

\

Given the specific type of ignorance captured by the system \texttt{HDI}, it seems inappropriate to directly address the example of William III as it is presented at the beginning of this section. Specifically, in the initial example, William III is not disbelievingly ignorant of both $p$ and $p \wedge (q \vee \neg q)$. However, we can present alternative examples of equivalent propositions that can be distinguished within \texttt{HDI}. Imagine a scenario where William III disbelieves that the war with England can be avoided ($p$), but he is not disbelievingly ignorant of $p \wedge (q \vee \neg q)$, because he does not grasp the concept of nuclear war. In this context, consider the following propositions.

\begin{itemize}

\item[(1$^{HDI}$)] William III disbelieves that $p$.

\item[(2$^{HDI}$)] William III disbelieves that $p \wedge (q \vee \neg q)$.

\end{itemize}

In this scenario, it appears that (1$^{HDI}$) is true, but not (2$^{HDI}$). This can be modelled in \texttt{HDI} as follows. Let $\mathcal{M} = (W, R, v, \mathcal{T})$, where $W = \{w, w'\}$, $R = \{(w,w')\}$, $v(p) = v(q) = \{w\}$, and $\mathcal{T} = (T, \oplus, t, \mathfrak{K})$, where $T =\{a, b\}$, $b \sqsubset a$, $t(q) = a$, and $t(p) = b = \mathfrak{K}$. In this model, it is clear that $\mathcal{M} \models p \leftrightarrow (p \wedge (q \vee \neg q))$. We also have that $\mathcal{M}, w \models I^{d} p$, because $\mathcal{M}, w \models p$, for all $w'$ s.t. they are not $w$ if $Rww'$ then $\mathcal{M}, w' \models \neg p$, and $t(p) \sqsubseteq \mathfrak{K}$; and $\mathcal{M}, w \models \neg I^{d} p \wedge (q \vee \neg q)$, because $t(q) \not \sqsubseteq \mathfrak{K}$. Thus, even though $p$ and $p \wedge (q \vee \neg q)$ are equivalent, while William III disbelieves $p$, he does not disbelieve $p \wedge (q \vee \neg q)$ because he does not grasp the topic $q$.

\section{Ignorance and Logical Omniscience}
\label{SecOmniscience}

Various works have explored the connection between hyperintensionality and the intricate issue of logical omniscience (see, e.g., \citet{Fagin1995}). Specifically, hyperintensionality has emerged as a natural solution to this well-known problem that plagues epistemic logic. Logical omniscience, as initially identified by \citet{Hintikka1962} in his influential book, appears typically in systems of epistemic logic that are closed under logical implication:

\begin{itemize}

\item[$(LI)$] if $\phi \rightarrow \psi$ and $K \phi$, then $K \psi$.

\end{itemize}

In a nutshell, the agents modelled via standard Kripke semantics are assumed to be perfect reasoners. However, this assumption becomes challenging when aiming to model the knowledge of real-world agents, such as humans or computers, given inherent practical limitations on their reasoning capacities. Recent approaches to logical omniscience highlight it as a broader and more pervasive phenomenon in epistemic logic. This stems from the observation that the operator $K$ in a system is closed under various principles, including logical equivalence, conjunction, or disjunction - not solely logical implication (see \citet{Fagin1995}). A general formulation of logical omniscience $(LO)$ can be as follows:

\begin{itemize}

\item[$(LO_{K})$] If an agent knows $\phi$, she knows all the consequences of $\phi$.

\end{itemize}

This formulation emphasizes the idealized nature of agents who possess knowledge of all consequences of their knowledge and, in particular, they know all the tautologies. Consequently, the situation poses a general challenge for the chosen formalism. On one hand, an agent modelled by the system should not logically derive more than is reasonable to expect from a real-world agent, considering limited capacities or attention for humans and constrained time and space resources for computers. On the other hand, the agent should retain the ability to logically derive everything we can reasonably expect her to derive in a given situation. Thus, solving logical omniscience is about finding the right balance between the level of idealization and the logical capabilities of the agents modelled by the system. While logical omniscience is commonly articulated in epistemic (or doxastic) terms, it prompts a broader inquiry into how this challenge materializes when one focuses on the representation of ignorance. From this perspective, a natural rephrasing of $(LO_{K})$ in terms of ignorance is as follows:

\begin{itemize}

\item[$(LO_{I})$] If an agent is not ignorant of $\phi$, she is not ignorant of all the consequences of $\phi$.

\end{itemize}

In order to analyze $(LO_{I})$ formally, we will consider three of its forms: closure under implication $(LO^{\rightarrow}_{I})$, omniscience rule $(LO^{NEC}_{I})$, and closure under equivalence $(LO^{RE}_{I})$.\footnote{For the sake of generality, we formulate these principles with $I$, which can be instantiated with $I^{w}$, $I^{u}$, and $I^{d}$.}

\begin{itemize}

\item[$(LO^{\rightarrow}_{I}$)] if $\models \phi \rightarrow \psi$, then $\models \neg I \phi \rightarrow \neg I \psi$;

\item[$(LO^{NEC}_{I}$)] if $\models \phi$, then $\models \neg I \phi$;

\item[$(LO^{RE}_{I}$)] if $\models \phi \leftrightarrow \psi$, then $\models \neg I \phi \leftrightarrow \neg I \psi$.

\end{itemize}

Let us first consider these principles in the systems \texttt{IW}, \texttt{IU}, and \texttt{DI} interpreted on standard Kripke frames.

\begin{prop}

$(LO^{\rightarrow}_{I})$ is not valid in \texttt{IW}, \texttt{IU}, and \texttt{DI}.

\end{prop}

\begin{proof}

We prove this proposition by constructing a countermodel.

Let $\mathcal{M} \models \phi \rightarrow \psi$ for any $\mathcal{M}$ (i.e. $\phi \rightarrow \psi$ is valid). We need to construct a model $\mathcal{M}'$ s.t. $\mathcal{M'}, w \not \models I \phi \rightarrow I \psi$, where $I$ stands for $I^{w}$, $I^{u}$, or $I^{d}$.

Let $\mathcal{M}' = (W, R, v)$, where $W = \{w, w'\}$, $R = \{(w,w'), (w,w)\}$, $v(p) = \{ \}$, and $v(q) = \{w\}$.

Clearly,  for and $\mathcal{M}$ we have: $\mathcal{M} \models (p \wedge \neg p) \rightarrow q$.

It is easy to check that $\mathcal{M}', w \models \neg I^{w} (p \wedge \neg p)$ and $\mathcal{M}, w \not \models \neg I^{w} q$, i.e., $\mathcal{M} \not \models \neg I^{w} (p \wedge \neg p) \rightarrow \neg I^{w} q$. Similarly, it is easy to check that $\mathcal{M} \not \models \neg I^{u} (p \wedge \neg p) \rightarrow \neg I^{u} q$ and $\mathcal{M} \not \models \neg I^{d} (p \wedge \neg p) \rightarrow \neg I^{d} q$.

\end{proof}

\begin{prop}

$(LO^{NEC}_{I})$ is valid in \texttt{IW}, \texttt{IU}.

\end{prop}

\begin{proof}

For the case of  \texttt{IW}, let (i) $\mathcal{M} \models \phi$ for an arbitrary $\mathcal{M}$ (i.e., $\phi$ is valid), and for some $w$ (ii) $\mathcal{M}, w \not \models \neg I^{w} \phi$. From (ii), we have that there exists $w'$ s.t. $Rww'$ and $\mathcal{M}, w' \models \neg \phi$, which contradicts (i). The case of \texttt{IU} is similar.

\end{proof}

\begin{prop}

$(LO^{NEC}_{I})$ is not valid in \texttt{DI}.

\end{prop}

\begin{proof}

We prove this proposition by constructing a countermodel.

Let $\mathcal{M} \models \phi $ for any $\mathcal{M}$ (i.e. $\phi$ is valid). We need to construct a model $\mathcal{M}'$ s.t. $\mathcal{M}', w \not \models \neg I^{d}\phi$.

Let $\mathcal{M}' = (W, R, v)$, where $W = \{w\}$, $R = \{\}$, $v(p) = \{ w\}$. 

It is clear that for any $\mathcal{M}$, it holds that $\mathcal{M} \models p \vee \neg p$, however $\mathcal{M}' \not \models \neg I^{d} (p \vee \neg p)$.

\end{proof}

\begin{prop}

$(LO^{RE}_{I})$ is valid in \texttt{IW}, \texttt{IU}, and \texttt{DI}.

\end{prop}

\begin{proof}

Let us consider only the case of \texttt{IW}, the other two cases are similar. 

Let (i) $\mathcal{M} \models \phi \leftrightarrow \psi$ for all $\mathcal{M}$ (i.e., $\phi \rightarrow \psi$ is valid) and (ii) $\mathcal{M}, w \models \neg I^{w} \phi$ for some $w$. The case (ii) is equivalent to say that either for all $w'$, if $Rww'$ then $\mathcal{M}, w' \models \phi$, or for all $w'$, if $Rww'$ then $\mathcal{M}, w' \models \neg \phi$. By (ii) this is equivalent to the fact that either for all $w'$, if $Rww'$ then $\mathcal{M}, w' \models \psi$, or for all $w'$, if $Rww'$ then $\mathcal{M}, w' \models \neg \psi$, that is $\mathcal{M}, w \models \neg I^{w} \psi$. Thus, $w$ being arbitrary world of an arbitrary model, we have $\mathcal{M} \models \neg I^{w} \phi \leftrightarrow \neg I^{w} \psi$ for all $\mathcal{M}$.

\end{proof}

Interestingly, $(LO^{\rightarrow}_{I})$ does not hold in \texttt{IW}, \texttt{IU}, and \texttt{DI}, and $(LO^{NEC}_{I})$ does not hold in \texttt{DI}. This means that standard Kripke semantics already provide several conditions for avoiding some forms of logical omniscience once it is expressed in terms of ignorance. However, this does not depend on the hyperintensional character of ignorance, but on the specific definition of each operator. We will now show that a hyperintensional setting permits us to avoid all three forms of logical omniscience.

\begin{prop}

$(LO^{\rightarrow}_{I})$, $(LO^{NEC}_{I})$, and $(LO^{RE}_{I})$ are not valid in \texttt{HIW}, \texttt{HIU}, and \texttt{HDI}.

\end{prop}

\begin{proof}

The case of $(LO^{\rightarrow}_{I})$ is straightforward from the corresponding results for \texttt{IW}, \texttt{IU}, and \texttt{DI}.

The case of $(LO^{NEC}_{I})$ in \texttt{HDI} is straightforward from the result for \texttt{DI}. 

The case of \texttt{HIW} and \texttt{HIU} is as follows. 

Let $\mathcal{M} \models \phi$ for all $\mathcal{M}$ (i.e., $\phi$ is valid). We need to construct a $\mathcal{M}'$ s.t. $\mathcal{M}', w \not \models \neg I \phi$, where $I$ stands for $I^{w}$ or $I^{u}$.

Let $\mathcal{M}' = (W, R, v, \mathcal{T})$, where $W = \{w\}$, $R = \{\}$, $v(p) = \{w\}$, and $\mathcal{T} = (T, \oplus, t, \mathfrak{K})$, where $T =\{a, b\}$, $b \sqsubset a$, $t(\phi) = a$, and $b = \mathfrak{K}$. It is clear that for any $\mathcal{M}$ we have $\mathcal{M} \models p \vee \neg p$. However, $\mathcal{M}', w \models I^{w} (p \vee \neg p)$ and $\mathcal{M}' \models I^{u} (p \vee \neg p)$, because $\mathfrak{K} \sqsubset t(p) = a$, that is $t(p) \not \sqsubseteq \mathfrak{K}$. Thus, $\mathcal{M}' \not \models \neg I^{w} (p \vee \neg p)$ and $\mathcal{M}' \not \models \neg I^{u} (p \vee \neg p)$.

The case of $(LO^{RE}_{I})$  is as follows. 

Let $\mathcal{M} \models \phi \leftrightarrow \psi$ for all $\mathcal{M}$ (i.e., $\phi \rightarrow \psi$ is valid). We construct a $\mathcal{M}'$ s.t. $\mathcal{M}', w \not \models \neg \phi \leftrightarrow \neg I \psi$, where $I$ stands for $I^{w}$, $I^{u}$, or $I^{d}$.

Let $\mathcal{M}' = (W, R, v, \mathcal{T})$, where $W = \{w\}$, $R = \{\}$, $v(p) = v(q) = \{w\}$, and $\mathcal{T} = (T, \oplus, t, \mathfrak{K})$, where $T =\{a, b\}$, $b \sqsubset a$, $t(\psi) = a$, and $t(\phi) = b = \mathfrak{K}$. It is clear that for any $\mathcal{M}$ it holds that $\mathcal{M} \models (p \vee \neg p) \leftrightarrow (q \vee \neg q)$. However, we have $\mathcal{M}', w \models \neg I^{w} (p \vee \neg p)$, since $\mathcal{M}' \models p \vee \neg p$ and  $t(p) \sqsubseteq \mathfrak{K}$, and $\mathcal{M}', w \models I^{w} (q \vee \neg q)$, since  $t(q) \not \sqsubseteq \mathfrak{K}$. Thus, $\mathcal{M}' \not \models \neg I^{w} (p \vee \neg p) \leftrightarrow \neg I^{w} (q \vee \neg q)$. The case of $I^{u}$ is the same. In case of $I^{d}$, we can easily check that $\mathcal{M}', w \models I^{d} (p \vee \neg p)$ and $\mathcal{M}', w \models \neg I^{d} (q \vee \neg q)$. Thus, $\mathcal{M}' \not \models \neg I^{d} (p \vee \neg p) \leftrightarrow \neg I^{d} (q \vee \neg q)$.

\end{proof}

The invalidity of the principles that lead to logical omniscience, as considered in this section, relies on the internal structure of the topic-sensitive models. Although the outcomes presented here may not be surprising, they highlight that the three formal frameworks for representing hyperintensional ignorance are well-behaved and align with natural expectations of ignorance as a hyperintensional notion.

\section{Conclusion}
\label{Conclusion}

In this work, we considered ignorance as a hyperintensional notion by supplementing the truth conditions in standard Kripke semantics with an account of topicality. Since there exist several types of ignorance, our starting intuition was that the specific form of ignorance one is willing to analyze might be sensitive to the topic under consideration in a given context. This led us to employ a topic-sensitive semantics to represent various forms of ignorance. This semantics is naturally hyperintensional, enabling the differentiation of logically or necessarily equivalent propositions. As a consequence, we were able to provide three sound and complete systems, each capturing a distinct sense in which an agent might be ignorant. Lastly, we reframed the well-known problem of logical omniscience in terms of ignorance and showed that certain forms of omniscience can be avoided within these settings.

A natural development of this line of research involves the exploration of additional forms of ignorance where topicality plays a crucial role in formalization. For instance, one can define \textit{complete ignorance whether}. As highlighted in the introduction, complete ignorance occurs when an agent fails to grasp the content of a proposition. From this perspective, we can characterize this kind of ignorance by  incorporating the necessary condition of the absence of grasping into the definition of ignorance whether. Such formalization would result in a distinct system compared to \texttt{HIW}. Similarly, alternative definitions involving grasping for $I^{u}$ and $I^{d}$ can be proposed, and we leave this exploration for future investigations.

\appendix

\section{Proofs for system \texttt{HIW}}

\subsection{Soundness of \texttt{HIW}}
\label{appHIWsound}

\textbf{Theorem \ref{theorHIWsound}.} The system $\mathtt{HIW}$ is sound.

\begin{proof}

The cases of (CPL) and (S5$_{\Box}$) are standard.

The case of ($I^{w} \leftrightarrow$) is as follows. Given $\mathcal{M} = (W, R, v, \mathcal{T})$, $\mathcal{M}, w \models I^{w} \phi$ for any $w \in \mathcal{M}$ iff $t(\phi) \not \sqsubseteq \mathfrak{K}$ or (there exist $w', w'' \in W$ such that $Rww'$, $Rww''$, $\mathcal{M}, w' \models \phi$, and $\mathcal{M}, w'' \models \neg \phi$) iff $\mathcal{M}, w \models I^{w} \neg \phi$.

For the case of (A1$_{I^{w}}$), let  $\mathcal{M}, w \models \neg I^{w}\phi \wedge \neg I^{w} \bar{\psi} \wedge \neg I^{w} \bar{\chi}$  for an arbitrary $w \in \mathcal{M}$. This means that (i) $t(\phi) \sqsubseteq \mathfrak{K}$, (ii) $t(\psi) \sqsubseteq \mathfrak{K}$, (iii) $t(\chi) \sqsubseteq \mathfrak{K}$ and either (iv) $\mathcal{M}, w' \models \phi$  for all $w'$ s.t. $Rww'$, or (v) $\mathcal{M}, w' \not \models \phi$ for all $w'$ s.t. $Rww'$. If (iv), then it is clear that (vi) $\mathcal{M}, w' \models \neg \phi \rightarrow \chi$  for all $w'$ s.t. $Rww'$. From (i), (iii), and (vi), we have (vii) $\mathcal{M}, w \models \neg I^{w}(\neg \phi \rightarrow \chi)$. If (v), then it is clear that (viii) $\mathcal{M}, w' \models \phi \rightarrow \psi$  for all $w'$ s.t. $Rww'$. From (i), (ii), and (viii), we have (ix) $\mathcal{M}, w \models \neg I^{w}(\phi \rightarrow \psi)$. Then, $\mathcal{M}, w \models \neg I^{w}(\phi \rightarrow \psi) \vee \neg I^{w}(\neg \phi \rightarrow \chi)$.

For the case of (A2$_{I^{u}}$), let $\mathcal{M}, w \models \neg I^{w} \bar{\phi}$, which means that $t(\phi) \sqsubseteq \mathfrak{K}$. Clearly, $\mathcal{M}, w' \models \phi \vee \neg \phi$ for any $w' \in \mathcal{M}$, and thus $\mathcal{M}, w \models \neg I^{w} (\phi \vee \neg \phi)$.

For the case of (A3$_{I^{w}}$), let $\mathcal{M}, w \models \neg I^{w} \phi \wedge \neg I^{w} \psi$ for an arbitrary $w \in \mathcal{M}$. This means that (i) $t(\phi) \sqsubseteq \mathfrak{K}$, (ii) $t(\psi) \sqsubseteq \mathfrak{K}$, and either (iii) $\mathcal{M}, w' \models \phi$ and $\mathcal{M}, w' \models \psi$ for all $w'$ s.t. $Rww'$, or (iv) $\mathcal{M}, w' \not \models \phi$ for all $w'$ s.t. $Rww'$ or $\mathcal{M}, w' \not \models \psi$ for all $w'$ s.t. $Rww'$. From (i) and (ii) we have (v) $t(\phi \wedge \psi) \sqsubseteq \mathfrak{K}$. Let (iii) be the case. Then, from (iii) and (v), we have $\mathcal{M}, w \models \neg I^{w} (\phi \wedge \psi)$. Let (iv) be the case. Then, for all $w'$ s.t., $Rww'$, we have $\mathcal{M}, w' \not \models \phi \wedge \psi$, and thus, taking into account (v), $\mathcal{M}, w \models \neg I^{w} (\phi \wedge \psi)$.

For the case of (A4$_{I^{w}}$), let (i) $\mathcal{M}, w\models \neg I^{w} \bar{\phi}$ for an arbitrary $w \in \mathcal{M}$. Assume that (ii) $\mathcal{M}, w \not \models \Box \neg I^{w} \bar{\phi}$. From (i) we have (iii) $t(\phi) \sqsubseteq \mathfrak{K}$. From (ii) we have that (iv) there exists $w' \in \mathcal{M}$ s.t. $\mathcal{M}, w' \models I \bar{\phi}$. From (iii), (iv) and the definition of $I^{w}$, we get that there exist $w'$ s.t. $\mathcal{M}, w' \models \neg \bar{\phi}$, which is a contradiction by construction of $\bar{\phi}$.

For the case of (A5$_{I^{w}}$), let $\mathcal{M}, w \models I^{w} \bar{\phi}$ for an arbitrary $w \in \mathcal{M}$. This means that $t(\phi) \not \sqsubseteq \mathfrak{K}$, which provides us with $\mathcal{M}, w \models I^{w} \phi$.

For the case of (A6$_{I^{w}}$), let (i) $\mathcal{M}, w \models \neg I^{w} (\phi \rightarrow \neg \psi) \wedge \neg I^{w} \phi \wedge \neg I^{w} (\psi \rightarrow \phi)$ for an arbitrary $w \in \mathcal{M}$. From  $\mathcal{M}, w \models \neg I^{w} (\phi \rightarrow \neg \psi)$ we have that $\mathcal{M}, w \models \neg I^{w} \bar{\phi}$ and $\mathcal{M}, w \models \neg I^{w} \bar{\psi}$, that is (ii) $t(\phi) \sqsubseteq \mathfrak{K}$ and (iii) $t(\psi) \sqsubseteq \mathfrak{K}$. From (i), (ii), (iii), we have (iv) either $\mathcal{M}, w' \models \phi \rightarrow \neg \psi$ for all $w'$ such that $Rww'$ or $\mathcal{M}, w' \not \models \phi \rightarrow \neg \psi$ for all $w'$ such that $Rww'$ for all $w'$ such that $Rww'$, and (v) either $\mathcal{M}, w' \models \phi$ for all $w'$ such that $Rww'$ or $\mathcal{M}, w' \not \models \phi$ for all $w'$ such that $Rww'$, and (vi) either $\mathcal{M}, w' \models \psi \rightarrow \phi$ for all $w'$ such that $Rww'$ or $\mathcal{M}, w' \not \models \psi \rightarrow \phi$ for all $w'$ such that $Rww'$. Assume that (vii) $\mathcal{M}, w \models I^{w} \psi$. From (vii) and (iii), this means that there exist $w''$ and $w'''$ s.t. $Rww''$, $Rww'''$, (viii) $\mathcal{M}, w'' \models \psi$, and (ix) $\mathcal{M}, w''' \models \neg \psi$. From (ix), (v) and (vi), we have that for all $w'$, s.t. $Rww'$, it holds that (x) $\mathcal{M}, w' \models \phi$. From (viii),(v) and (iv), we have that for all $w'$ s.t. $Rww'$ it holds that $\mathcal{M}, w' \models \neg \phi$, which contradicts (x). Thus, $\mathcal{M}, w \models \neg I^{w} \psi$.

For the case of (A7$_{I^{w}}$), let $\mathcal{M}, w \models \neg I^{w} \bar{\phi}$ for an arbitrary $w \in \mathcal{M}$, which means that $t(\phi) \sqsubseteq \mathfrak{K}$. By definition of $\bar{\phi}$ and of $t$, this means that $t(\psi) \sqsubseteq \mathfrak{K}$ for any $\psi$ s.t. $Var(\psi) \subseteq Var (\phi) $, that is $\mathcal{M}, w \models \neg I^{w} \bar{\psi}$, if $Var(\psi) \subseteq Var(\phi)$.

For the case of (NEC$_{I^{w}}$), assume that (i) $\mathcal{M} \models \phi$ for an arbitrary $\mathcal{M}$ (i.e., $\phi$ is valid) and let (ii) $\mathcal{M}, w \models \neg I^{w} \bar{\phi}$ for an arbitrary $w \in \mathcal{M}$. From (ii) we have (iii) $t(\phi) \sqsubseteq \mathfrak{K}$. From (i) we have $\mathcal{M}, w \models \phi$ and for all $w'$ if $Rww'$ then $\mathcal{M}, w' \models \phi$. This means, taking into account (ii), that $\mathcal{M}, w \models \neg I^{w} \phi$. 

For the case of (RE$_{I^{w}}$), assume that (i) $\mathcal{M} \models \phi \leftrightarrow \psi$ for an arbitrary $\mathcal{M}$ and let (ii) $\mathcal{M}, w \models \neg I^{w} \bar{\phi} \wedge \neg I^{w} \bar{\psi}$ for an arbitrary $w \in \mathcal{M}$. From (ii) we have (iii) $t(\phi) \sqsubseteq \mathfrak{K}$ and (iv) $t(\psi) \sqsubseteq \mathfrak{K}$. Let (v) $\mathcal{M}, w \models \neg I^{w} \phi$. From (v) and (iii), we have  for all $w'$ s.t. $Rww'$ either (vi) $\mathcal{M}, w' \models \phi$ or (vii) $\mathcal{M}, w' \models \neg \phi$. If (vi), then from (i) we have that for all $w'$ s.t. $Rww'$, we have $\mathcal{M}, w' \models \psi$, that is (taking into account (iv)), $\mathcal{M}, w \models \neg I^{w} \psi$. The reasoning for (vii), and the backwards direction is the same. Thus, $\mathcal{M}, w \models \neg I^{w} \phi \leftrightarrow \neg I^{w} \psi$.

\end{proof}

\subsection{Completeness of \texttt{HIW}}
\label{appHIWcomplete}

For the completeness proof, we also need the following proposition which is a generalization of (A6$_{I^{w}}$), and is an analogue of Proposition 4.4 in \citep{Fan2015}. However, notice that its proof is different from the one of Proposition 4.4.

\begin{prop}
\label{genA6IW}

For all $k \geq 1:$

\begin{center}

$\vdash (\neg I^{w} (\bigwedge^{k}_{j=1} \phi_{j} \rightarrow \neg \psi) \wedge \bigwedge^{k}_{j=1} \neg I^{w} \phi_{j} \wedge \bigwedge^{k}_{j=1} \neg I^{w}(\psi \rightarrow \phi_{j})) \rightarrow \neg I^{w} \psi$

\end{center}

\end{prop}

\begin{proof}

\

\begin{enumerate}

\item $(\neg I^{w} (\bigwedge^{k}_{j=1} \phi_{j} \rightarrow \neg \psi) \wedge \neg I^{w}\bigwedge^{k}_{j=1} \phi_{j}  \wedge \neg I^{w}(\psi \rightarrow \bigwedge^{k}_{j=1} \phi_{j} )) \rightarrow \neg I^{w} \psi$ (by (A6$_{I^{w}}$));

\item $\bigwedge^{k}_{j=1} \neg I^{w} \phi_{j} \rightarrow \neg I^{w} \bigwedge^{k}_{j=1} \phi_{j}$ (from (A3$_{I^{w}}$) applied $k - 1$ times);

\item $\bigwedge^{k}_{j=1} \neg I^{w}(\psi \rightarrow \phi_{j}) \rightarrow \neg I^{w} \bigwedge^{k}_{j=1} (\psi \rightarrow \phi_{j})$ (from (A3$_{I^{w}}$) applied $k-1$ times);

\item $\bigwedge^{k}_{j=1} (\psi \rightarrow \phi_{j}) \leftrightarrow (\psi \rightarrow \bigwedge^{k}_{j=1} \phi_{j})$ (TAUT);

\item $(\neg I^{w} \overline{\bigwedge^{k}_{j=1} (\psi \rightarrow \phi_{j})} \wedge \neg I^{w} \overline{(\psi \rightarrow \bigwedge^{k}_{j=1} \phi_{j})}) \rightarrow (\neg I^{w} \bigwedge^{k}_{j=1} (\psi \rightarrow \phi_{j}) \leftrightarrow \neg I^{w} (\psi \rightarrow \bigwedge^{k}_{j=1} \phi_{j}))$ (from 4, by (RE$_{I^{w}}$));

\item $\neg I^{w} \overline{\bigwedge^{k}_{j=1} (\psi \rightarrow \phi_{j})} \leftrightarrow \neg I^{w} \overline{(\psi \rightarrow \bigwedge^{k}_{j=1} \phi_{j})}$ (by definition of $\bar{\chi_{1}}$ and $\bar{\chi_{2}}$ for any $\chi_{1}$ and $\chi_{2}$ in which occur only the same propositional variables);

\item $\neg I^{w} \bigwedge^{k}_{j=1} (\psi \rightarrow \phi_{j}) \rightarrow \neg I^{w} \overline{\bigwedge^{k}_{j=1} (\psi \rightarrow \phi_{j})}$ (from (A5$_{I^{w}}$) by contraposition);

\item $\neg I^{w} \bigwedge^{k}_{j=1} (\psi \rightarrow \phi_{j}) \rightarrow (\neg I^{w} \overline{\bigwedge^{k}_{j=1} (\psi \rightarrow \phi_{j})} \wedge \neg I^{w} \overline{(\psi \rightarrow \bigwedge^{k}_{j=1} \phi_{j})})$ (from 6, 7, by CPL);

\item $\neg I^{w} \bigwedge^{k}_{j=1} (\psi \rightarrow \phi_{j}) \rightarrow (\neg I^{w} \bigwedge^{k}_{j=1} (\psi \rightarrow \phi_{j}) \leftrightarrow \neg I^{w} (\psi \rightarrow \bigwedge^{k}_{j=1} \phi_{j}))$ (from 5, 8, by transitivity);

\item $\neg I^{w} \bigwedge^{k}_{j=1} (\psi \rightarrow \phi_{j}) \rightarrow \neg I^{w} (\psi \rightarrow \bigwedge^{k}_{j=1} \phi_{j})$ (from 9, by CPL);

\item $\bigwedge^{k}_{j=1} \neg I^{w}(\psi \rightarrow \phi_{j}) \rightarrow \neg I^{w} (\psi \rightarrow \bigwedge^{k}_{j=1} \phi_{j})$ (from 3, 10, by transitivity);

\item $\neg I^{w} (\bigwedge^{k}_{j=1} \phi_{j} \rightarrow \neg \psi) \rightarrow \neg I^{w} (\bigwedge^{k}_{j=1} \phi_{j} \rightarrow \neg \psi)$ (TAUT);

\item $(\neg I^{w} (\bigwedge^{k}_{j=1} \phi_{j} \rightarrow \neg \psi) \wedge \bigwedge^{k}_{j=1} \neg I^{w} \phi_{j} \wedge \bigwedge^{k}_{j=1} \neg I^{w}(\psi \rightarrow \phi_{j})) \rightarrow (\neg I^{w} (\bigwedge^{k}_{j=1} \phi_{j} \rightarrow \neg \psi) \wedge \neg I^{w}\bigwedge^{k}_{j=1} \phi_{j}  \wedge \neg I^{w}(\psi \rightarrow \bigwedge^{k}_{j=1} \phi_{j} ))$ (from 2, 11, 12, by CPL);

\item $(\neg I^{w} (\bigwedge^{k}_{j=1} \phi_{j} \rightarrow \neg \psi) \wedge \bigwedge^{k}_{j=1} \neg I^{w} \phi_{j} \wedge \bigwedge^{k}_{j=1} \neg I^{w}(\psi \rightarrow \phi_{j})) \rightarrow \neg I^{w} \psi$ (from 1, 13, by transitivity).

\end{enumerate}

\end{proof}

Completeness is proved by constructing a canonical model. We define maximal consistent sets in a standard way. A set $\Gamma \subseteq \mathcal{L}_{I^{w}}$ is consistent if $\Gamma \not \vdash^{\texttt{ HIW}} \perp$; $\Gamma$ is maximal if there is not $\Gamma' \subseteq \mathcal{L}_{I^{w}}$ s.t. $\Gamma \subset \Gamma'$ and $\Gamma' \not \vdash^{ \texttt{HIW}} \perp$. The following lemma can be proven in a standard way.

\begin{lemma}[Lindenbaum's lemma]
\label{LindIW}

Every consistent set of $\texttt{HIW}$ can be extended to a maximally consistent one.

\end{lemma}

Let $\mathcal{X}^{C}$ be the set of all $\mathtt{HIW}$-maximally consistent sets. For each $\Gamma \in \mathcal{X}^{C}$, we define:

\begin{itemize}

\item $\Gamma[\Box] := \{ \phi \mid \Box \phi \in \Gamma\}$;

\item $\Gamma[I^{w}] := \{\phi \mid \neg I^{w} \phi \wedge \neg I^{w} (\psi \rightarrow \phi) \in \Gamma\}$.

\end{itemize}

We define $\sim$ on $\mathcal{X}^{C}$ as 

\begin{center}

$\Gamma \sim \Delta$ iff $\Gamma[\Box] \subseteq \Delta$.

\end{center}

By definition, and since $\Box$ is an S5 modality, it is easy to verify that $\sim$ is an equivalence relation. By the same reason the following proposition can be proved straightforwardly.

\begin{prop}
\label{equivIW}

For any two maximally consistent sets $\Gamma$ and $\Delta$ s.t. $\Gamma \sim \Delta$, $\Gamma[\Box] = \Delta[\Box]$.

\end{prop}

Now we are able to define canonical model for a given set of formulas.

\begin{defi}[Canonical Model for $\Gamma_{0}$]
\label{canmodIW}

Given a maximal consistent set $\Gamma_{0}$ of $\texttt{HIW}$, the canonical model for $\Gamma_{0}$ is a tuple $\mathcal{M}^{C} = \{W^{C}, R^{C}, v^{C}, \mathcal{T}^{C}\}$, where 

\begin{itemize}

\item $W^{C} = \{ \Gamma \in \mathcal{X}^{C} \mid \Gamma_{0} \sim \Gamma\}$;

\item $R^{C} \subseteq W^{C} \times W^{C}$ s.t. for all $\Gamma, \Delta \in W^{C}$

\begin{center}

(1) $\neg I^{w} \bar{\psi} \wedge I^{w} \psi \in \Gamma$ and

(2) $\Gamma R^{C} \Delta$ iff $\Gamma[I^{w}] \subseteq \Delta$,

where $\Gamma[I^{w}] := \{\phi \mid \neg I^{w} \phi \wedge \neg I^{w} (\psi \rightarrow \phi) \in \Gamma\}$;

\end{center}

\item $v^{C}(p) = \{\Gamma \in W^{C} \mid p \in \Gamma\}$;

\item $\mathcal{T}^{C}$ is such that

\begin{itemize}

\item $T^{C}= \{a, b\}$ where $a = \{p \in \mathtt{Prop} \mid I^{w} \bar{p} \in \Gamma_{0}\}$ and $b = \{ p \in \mathtt{Prop} \mid \neg I^{w} \bar{p} \in \Gamma_{0}\}$;

\item $\oplus^{C} : T^{C} \times T^{C} \mapsto T^{C}$ s.t. $a \oplus^{C} a = a$, $b \oplus^{C} b = b$, $a \oplus^{C} b = b \oplus^{C} a = a$;

\item $\mathfrak{K}^{C} = b$;

\item $t^{C}: \mathtt{Prop} \mapsto T^{C}$ s.t. for all $c \in T^{C}$ and $p \in \mathtt{Prop} : t^{C}(p) = c$ iff $p \in c$, and $t^{C}$ extends to the whole language by $t^{C}(\phi) = \oplus^{C}\{t^{C}(p) \mid p \in Var(\phi)\}$.

\end{itemize}

\end{itemize}

\end{defi}

To prove the Truth Lemma, the following propositions are useful.

\begin{prop}
\label{useful1}

For all $\phi \in \mathcal{L}_{I^{w}}$ and for any $\Gamma \in W^{C}$: $\neg I^{w} \bar{\phi} \in \Gamma$ iff $\forall p \in Var(\phi)$: $\neg I^{w} \bar{p} \in \Gamma$.

\end{prop}

\begin{proof}

The proof follows closely the proof of Lemma 9 \citep{Ozgun2023}.

\begin{itemize}

\item[$(\Rightarrow)$] Let $\neg I^{w} \bar{\phi} \in \Gamma$. Then, by (A7$_{I^{w}}$), $\neg I^{w} \overline{p} \in \Gamma$ for all $p \in Var(\phi)$.

\item[$(\Leftarrow)$] Let $Var(\phi) = \{p_{1}, ..., p_{n}\}$. Then, $\bar{\phi} = \bar{p_{1}} \wedge ... \wedge \bar{p_{n}}$. Let $\neg I^{w} \bar{p_{i}} \in \Gamma$ for all $p_{i} \in \{p_{1}, ..., p_{n}\}$. Then, $\bigwedge_{i \leq n} \neg I^{w} \bar{p_{i}} \in \Gamma$, because $\Gamma$ is a maximal consistent set. By (A3$_{I^{w}}$), $\neg I^{w} \bigwedge_{i \leq n} \bar{p_{i}} \in \Gamma$, that is $\neg I^{w} \bar{\phi} \in \Gamma$.

\end{itemize}

\end{proof}

\begin{prop}
\label{notIwtopic}

Given a canonical topical-sensitive model $\mathcal{M}^{C} =\{W^{C}, R^{C}, v^{C}, \mathcal{T}^{C}\}$, for any $\Gamma \in W^{C}$, and $\phi \in \mathcal{L}_{I^{w}}$:

\begin{center}

$\neg I^{w} \bar{\phi} \in \Gamma$ iff $t^{C}(\phi) \sqsubseteq \mathfrak{K}^{C}$.

\end{center}

\end{prop}

\begin{proof}

The proof is the same as of Corollary 10 \citep{Ozgun2023}. We just replace $K$ by $\neg I^{w}$, Lemma 9 by Proposition \ref{useful1}, and Ax3$_{K}$ by (A4$_{I^{w}}$).

\end{proof}

The definition of canonical relation $R^{C}$ and the proof of Truth Lemma follow closely \cite[p. 84-85]{Fan2015}.

\begin{lemma}[Truth Lemma]
\label{truthIW}

Let $\mathcal{M}^{C} = \{W^{C}, R^{C}, v^{C}, \mathcal{T}^{C}\}$ be the canonical model for $\Gamma_{0}$. Then, for all $\phi \in \mathcal{L}_{I^{w}}$ and $\Gamma \in W^{C}$, we have $\mathcal{M}^{C}, \Gamma \models \phi$ iff $\phi \in \Gamma$.

\end{lemma}

\begin{proof}

The proof is by induction on the complexity of $\phi$. The cases for propositional variables, boolean operators and $\Box$ are standard. The case of $\phi := I^{w} \psi$ is as follows.

\begin{itemize}

\item[($\Rightarrow$)] Assume $\mathcal{M}^{C}, \Gamma \models I^{w} \psi$. This means that either $t(\psi) \not \sqsubseteq \mathfrak{K}^{C}$, or there exist $\Gamma'$ and $\Gamma''$ such that $\Gamma R^{C} \Gamma'$, $\Gamma R^{C} \Gamma''$, $\mathcal{M}^{C}, \Gamma' \models \psi$, and $\mathcal{M}^{C}, \Gamma'' \models \neg \psi$. If $t(\psi) \not \sqsubseteq \mathfrak{K}^{C}$, then by Prop. \ref{notIwtopic}, $\neg I^{w} \bar{\psi} \not \in \Gamma$, that is $I^{w} \bar{\psi} \in \Gamma$, and thus, by (A5$_{I^{w}}$) $I^{w} \psi \in \Gamma$ (for which we searched).

Let $t^{C}(\psi) \sqsubseteq \mathfrak{K}^{C}$. By induction hypothesis we have (i) $\psi \in \Gamma'$ and (ii) $\neg \psi \in \Gamma''$. By definition of $\Gamma R^{C} \Gamma'$, there exist $\chi$ s.t. (iii) $\neg I^{w}\bar{\chi} \wedge I^{w} \chi \in \Gamma$ and (iv) for all $\theta$, if $\neg I^{w} \theta \wedge \neg I^{w} (\chi \rightarrow \theta) \in \Gamma$, then $\theta \in \Gamma'$. Let (v) $\neg I^{w} \psi \in \Gamma$. From (v) and (i), we have (vi) $\neg I^{w} \neg \psi \in \Gamma$ and (vii) $\neg \psi \not \in \Gamma'$. From (iv), (vi), and (vii), we have (viii) $I^{w}( \chi \rightarrow \neg \psi) \in \Gamma$. From (viii), (iii) ($\neg I^{w} \bar{\chi} \in \Gamma$, our assumption that $\neg I^{w} \bar{\psi} \in \Gamma$, and (RE$_{I^{w}}$), we have (ix) $I^{w}(\psi \rightarrow \neg \chi) \in \Gamma$.Then, from (ii), in the same way as for the case of (i), we get (x) $I^{w}(\neg \psi \rightarrow \chi') \in \Gamma$ for some $\chi'$. From $\neg I^{w} \bar{\chi} \in \Gamma$, $\neg I^{w} \bar{\chi'} \in \Gamma$,  (ix), (x), and (A1$_{I^{w}}$), we have $I^{w} \psi \in \Gamma$, which contradicts (v).

\item[($\Leftarrow$)] Assume that $I^{w} \psi \in \Gamma$. If $t^{C}(\psi) \not \sqsubseteq \mathfrak{K}^{C}$, then, by the definition of $I^{w}$, we have $\mathcal{M}^{C}, \Gamma \models I^{w}\psi$. If $t^{C}(\psi) \sqsubseteq \mathfrak{K}^{C}$, then, by Prop. \ref{notIwtopic}, we have $\neg I^{w} \bar{\psi} \in \Gamma$. Thus, we have $\neg I^{w} \bar{\psi} \wedge I^{w} \psi \in \Gamma$ and $\neg I^{w} \bar{\neg \psi} \wedge I^{w} \neg \psi \in \Gamma$. Then, we need to construct $\Gamma' , \Gamma'' \in W^{C}$, s.t. $\Gamma R^{C} \Gamma'$, $\Gamma R^{C} \Gamma''$, $\psi \in \Gamma'$, and $\neg \psi \in \Gamma''$. This contributes to show that:

\begin{enumerate}

\item[(1)] $\{\chi \mid \neg I^{w} \chi \wedge \neg I^{w}(\psi \rightarrow \chi) \in \Gamma\} \cup \{\psi\}$ is consistent. 

\item[(2)] $\{\chi \mid\neg I^{w} \chi \wedge \neg I^{w}(\neg \psi \rightarrow \chi) \in \Gamma\} \cup \{\neg \psi\}$ is consistent. 

\end{enumerate}

For the case (1), assume it is inconsistent, that is, (i) there exist $\chi_{1}, ..., \chi_{n}$ s.t. $\vdash \chi_{1} \wedge ... \wedge \chi_{n} \rightarrow \neg \psi$ and (ii) $\neg I^{w} \chi_{k} \wedge \neg I^{w}(\psi \rightarrow \chi_{k}) \in \Gamma$ for all $k \in [1, n]$. From (ii) and (A5$_{I^{w}}$), we get $\neg I^{w} \bar{\chi_{k}} \in \Gamma$ for all $k \in [1, n]$. Then, from $(NEC_{I^{w}})$, $\neg I^{w}(\chi_{1} \wedge ... \wedge \chi_{n} \rightarrow \neg \psi) \in \Gamma$. By Prop. \ref{genA6IW}, we have $\neg I^{w} \psi \in \Gamma$, which is a contradiction.

From (1), the definition of $R^{C}$, and Lemma \ref{LindIW}, we conclude that there exist $\Gamma'$ s.t. $\Gamma R^{C} \Gamma'$ and $\psi \in \Gamma'$.

For the case (2) the reasoning is the same. 
Thus, there exist $\Gamma''$ s.t. $\Gamma R^{C} \Gamma''$ and $\psi \in \Gamma''$. By induction hypothesis and the definition of $I^{w}$, this means that $\mathcal{M}^{C}, \Gamma \models I^{w} \psi$.

\end{itemize}

\end{proof}

\textbf{Theorem \ref{theorHIWcomplete}}. The system $\mathtt{HIW}$ is complete.

\begin{proof}

Let $\not \vdash \phi$. Then $\{\neg \phi\}$ is consistent. By Lemma \ref{LindIW}, there exists a maximal consistent set $\Gamma_{0}$ s.t. $\phi \not \in \Gamma_{0}$. By Lemma \ref{truthIW}, we have $\mathcal{M}^{C}, \Gamma_{0} \not \models \phi$.

\end{proof}

\section{Proofs for system \texttt{HIU}}

\subsection{Soundness of \texttt{HIU}}
\label{appHIUsound}

\begin{theor}

The system $\mathtt{HIU}$ is sound.

\end{theor}

\begin{proof}

The cases of (CPL) and (S5$_{\Box}$) are standard.

For the case of (A1$_{I^{u}}$), let $\mathcal{M}, w \models I^{u} \phi \wedge \neg I^{u}\bar{\phi}$ for an arbitrary $w \in \mathcal{M}$. Then, $t(\phi) \sqsubseteq \mathfrak{K}$, and thus $\mathcal{M}, w \models \phi$.

For the case of (A2$_{I^{u}}$), let $\mathcal{M}, w \models \neg I^{u} \bar{\phi}$, which means that $t(\phi) \sqsubseteq \mathfrak{K}$. Clearly, $\mathcal{M}, w' \models \phi \vee \neg \phi$ for any $w' \in \mathcal{M}$, and thus $\mathcal{M}, w \models \neg I^{u} (\phi \vee \neg \phi)$.

For the case of (A3$_{I^{u}}$), let $\mathcal{M}, w \models \neg I^{u} \phi \wedge \neg I^{u} \psi$ for an arbitrary $w \in \mathcal{M}$. This means that $t(\phi) \sqsubseteq \mathfrak{K}$, $t(\psi) \sqsubseteq \mathfrak{K}$, and either (i) $\mathcal{M}, w \not \models \phi$ or (ii) for all $w'$ s.t. $Rww'$, $\mathcal{M}, w' \models \phi$, and either (iii) $\mathcal{M}, w \not \models \psi$, or (iii) for all $w'$ s.t. $Rww'$, $\mathcal{M}, w' \models \psi$. If (i) or (iii), we have $\mathcal{M}, w \not \models \phi \wedge \psi$, and thus $\mathcal{M}, w \models \neg I^{u} (\phi \wedge \psi)$. If (ii) and (iv), then for all $w'$ s.t. $Rww'$, $\mathcal{M}, w' \models \phi \wedge \psi$, and thus $\mathcal{M}, w \models \neg I^{u} (\phi \wedge \psi)$.

For the case of (A4$_{I^{u}}$), let (i) $\mathcal{M}, w \models \neg I^{u} \bar{\phi}$ for an arbitrary $w \in \mathcal{M}$. Assume that (ii) $\mathcal{M}, w \not \models \Box \neg I^{u} \bar{\phi}$. From (i) we have that (iii) $t (\phi) \sqsubseteq \mathfrak{K}$. From (ii) we have that (iv) there exists $w' \in \mathcal{M}$ s.t. $\mathcal{M}, w' \models I^{u} \bar{\phi}$. From (iii), (iv) and the definition of $I^{u}$, we get that there exist $w'$ s.t. $\mathcal{M}, w' \models \neg \bar{\phi}$, which is a contradiction by construction of $\bar{\phi}$.

For the case of (A5$_{I^{u}}$), let $\mathcal{M}, w \models I^{u} \bar{\phi}$ for an arbitrary $w \in \mathcal{M}$. This means that $t(\phi) \not \sqsubseteq \mathfrak{K}$, and thus $\mathcal{M}, w \models I^{u} \phi$.

For the case of (A6$_{I^{u}}$), let $\mathcal{M}, w \models \phi \wedge \neg I^{u} \bar{\phi} \wedge \neg I^{u}\bar{\psi} \wedge I^{u}(\phi \vee \psi)$ for an arbitrary $w \in \mathcal{M}$. This means that (i) $\mathcal{M}, w \models \phi$, (ii) $t(\phi) \sqsubseteq \mathfrak{K}$, (iii) $t(\psi) \sqsubseteq \mathfrak{K}$, (iv) $\mathcal{M}, w \models \phi \vee \psi$, and (v) there exist $w'$ s.t. $Rww'$ and $\mathcal{M}, w' \not \models \phi$ and $\mathcal{M}, w' \not \models \psi$. Thus, from (i) and (v), we get $\mathcal{M}, w \models I^{u} \phi$.

For the case of (A7$_{I^{u}}$), let $\mathcal{M}, w \models \neg I^{u} (\phi \vee \psi) \wedge \neg I^{u}( \phi \rightarrow \psi)$ for an arbitrary $w \in \mathcal{M}$. This means that either (i) $\mathcal{M}, w \models \neg (\phi \vee \psi)$, or (ii) for all $w'$ if $Rww'$ then $\mathcal{M}, w' \models \phi \vee \psi$, and either (iii) $\mathcal{M}, w \models \neg (\phi \rightarrow \psi)$, or (iv) for all $w'$, if $Rww'$ then $\mathcal{M}, w' \models \phi \rightarrow \psi$. Assume that $\mathcal{M}, w \models I^{u} \psi$. Then, (v) $\mathcal{M}, w \models \psi$ and (vi) there exist $w''$ s.t.w $Rww''$ and $\mathcal{M}, w'' \models \neg \psi$. From (v), (i) and (iii) are not the cases. Thus, conditions (ii) and (iv) hold. From (vi) and (ii), we have (vii) $\mathcal{M}, w'' \models \phi$. From (vii) and (iv), we have $\mathcal{M}, w'' \models \psi$, which is a contradiction.

For the case of (A8$_{I^{u}}$), let $\mathcal{M}, w \models \neg I^{u} \bar{\phi}$ for an arbitrary $w \in \mathcal{M}$, which means that $t(\phi) \sqsubseteq \mathfrak{K}$. By definition of $\bar{\phi}$ and of $t$, this means that $t(\psi) \sqsubseteq \mathfrak{K}$ for any $\psi$ s.t. $Var(\psi) \subseteq Var (\phi) $, that is $\mathcal{M}, w \models \neg I^{u} \bar{\psi}$, if $Var(\psi) \subseteq Var(\phi)$.

For the case of (NEC$_{I^{u}}$), assume that (i) $\mathcal{M} \models \phi$ for an arbitrary $\mathcal{M}$ (i.e., $\phi$ is valid) and let (ii) $\mathcal{M}, w \models \neg I^{u} \bar{\phi}$ for an arbitrary $w \in \mathcal{M}$. From (ii) we have (iii) $t(\phi) \sqsubseteq \mathfrak{K}$. From (i) we have $\mathcal{M}, w \models \phi$ and for all $w'$ if $Rww'$ then $\mathcal{M}, w' \models \phi$. This means, taking into account (ii), that $\mathcal{M}, w \models \neg I^{u} \phi$. 

For the case of (RE$_{I^{u}}$), assume that (i) $\mathcal{M} \models \phi \leftrightarrow \psi$ for an arbitrary $\mathcal{M}$ and let (ii) $\mathcal{M}, w \models \neg I^{u} \bar{\phi} \wedge \neg I^{u} \bar{\psi}$ for an arbitrary $w \in \mathcal{M}$. From (ii) we have (iii) $t(\phi) \sqsubseteq \mathfrak{K}$ and (iv) $t(\psi) \sqsubseteq \mathfrak{K}$. Let (v) $\mathcal{M}, w \models \neg I^{u} \phi$. From (v) and (iii), we have either (vi) $\mathcal{M}, w \models \neg \phi$, or (vii) for all $w'$ if $Rww'$ then $\mathcal{M}, w' \models \phi$. If (vi) is the case, then from (i) and (vi) we have $\mathcal{M}, w \models \neg \psi$ and thus, from (iv), $\mathcal{M}, w \models \neg I^{u} \psi$. If (vii) is the case, then, from (i), for all $w'$ if $Rww'$ then $\mathcal{M}, w' \models \psi$, and thus $\mathcal{M}, w \models \neg I^{u} \psi$ (taking into account (iv)). The reasoning for the backwards direction is the same. Thus, $\mathcal{M}, w \models \neg I^{u} \leftrightarrow \neg I^{u} \psi$.

\end{proof}

\subsection{Completeness of \texttt{HIU}}
\label{appHIUcomplete}

For the completeness proof, we need the following proposition which is a generalization of (A7$_{I^{u}}$).

\begin{prop}
\label{genA7IU}

For all $k \geq 1:$

\begin{center}

$\vdash (\bigwedge^{k}_{j=1} \neg I^{u} ( \phi_{j} \vee \psi) \wedge  \neg I^{u} (\bigwedge^{k}_{j=1}\phi_{j} \rightarrow \psi)) \rightarrow \neg I^{u} \psi$

\end{center}

\end{prop}

\begin{proof}

\

\begin{enumerate}

\item $(\neg I^{u} (\bigwedge^{k}_{j=1} \phi_{j} \vee \psi) \wedge \neg I^{u} (\bigwedge^{k}_{j=1} \phi_{j} \rightarrow \psi)) \rightarrow \neg I^{u} \psi$ (from (A7$_{I^{u}}$));

\item $\bigwedge^{k}_{j=1}  \neg I^{u}( \phi_{j} \vee \psi) \rightarrow \neg I^{u} \bigwedge^{k}_{j=1} (\phi_{j} \vee \psi)$ (from (A3$_{I^{u}}$) applied $k-1$ times);

\item $\bigwedge^{k}_{j=1} (\phi_{j} \vee \psi) \leftrightarrow (\bigwedge^{k}_{j=1} \phi_{j}  \vee \psi)$ (TAUT);

\item $(\neg I^{u} \overline{\bigwedge^{k}_{j=1} (\phi_{j} \vee \psi)} \wedge \neg I^{u} \overline{(\bigwedge^{k}_{j=1} \phi_{j}  \vee \psi)}) \rightarrow (\neg I^{u} \bigwedge^{k}_{j=1} (\phi_{j} \vee \psi) \leftrightarrow \neg I^{u} (\bigwedge^{k}_{j=1} \phi_{j}  \vee \psi))$ (from 3, by (RE$_{I^{u}}$));

\item $\neg I^{u} \overline{\bigwedge^{k}_{j=1} (\phi_{j} \vee \psi)} \leftrightarrow \neg I^{u} \overline{(\bigwedge^{k}_{j=1} \phi_{j}  \vee \psi)}$ (by definition of $\bar{\chi_{1}}$ and $\bar{\chi_{2}}$ for any $\chi_{1}$ and $\chi_{2}$ in which occur only the same propositional variables);

\item $\neg I^{u} \bigwedge^{k}_{j=1} (\phi_{j} \vee \psi) \rightarrow \neg I^{u} \overline{\bigwedge^{k}_{j=1} (\phi_{j} \vee \psi)}$ (from (A5$_{I^{w}}$) and contraposition);

\item $\neg I^{u} \bigwedge^{k}_{j=1} (\phi_{j} \vee \psi) \rightarrow (\neg I^{u} \overline{\bigwedge^{k}_{j=1} (\phi_{j} \vee \psi)} \wedge \neg I^{u} \overline{(\bigwedge^{k}_{j=1} \phi_{j}  \vee \psi))}$ (from 5, 6, by CPL);

\item $\neg I^{u} \bigwedge^{k}_{j=1} (\phi_{j} \vee \psi) \rightarrow  (\neg I^{u} \bigwedge^{k}_{j=1} (\phi_{j} \vee \psi) \leftrightarrow \neg I^{u} (\bigwedge^{k}_{j=1} \phi_{j}  \vee \psi))$ (from 4, 7, by transitivity);

\item $\neg I^{u} \bigwedge^{k}_{j=1} (\phi_{j} \vee \psi) \rightarrow \neg I^{u} (\bigwedge^{k}_{j=1} \phi_{j}  \vee \psi))$ (from 8 by CPL);

\item $\bigwedge^{k}_{j=1}  \neg I^{u}( \phi_{j} \vee \psi) \rightarrow \neg I^{u} (\bigwedge^{k}_{j=1} \phi_{j}  \vee \psi)$ (from 2, 9, by transitivity);

\item $\neg I^{u} (\bigwedge^{k}_{j=1} \phi_{j} \rightarrow \psi) \rightarrow \neg I^{u} (\bigwedge^{k}_{j=1} \phi_{j} \rightarrow \psi)$ (TAUT);

\item $(\bigwedge^{k}_{j=1}  \neg I^{u}( \phi_{j} \vee \psi) \wedge \neg I^{u} (\bigwedge^{k}_{j=1} \phi_{j}\rightarrow \psi)) \rightarrow (\neg I^{u} (\bigwedge^{k}_{j=1} \phi_{j} \vee \psi) \wedge \neg I^{u} (\bigwedge^{k}_{j=1} \phi_{j} \rightarrow \psi))$ (from 10, 11, by CPL);

\item $(\bigwedge^{k}_{j=1} \neg I^{u} ( \phi_{j} \vee \psi) \wedge  \neg I^{u} (\bigwedge^{k}_{j=1}\phi_{j} \rightarrow \psi)) \rightarrow \neg I^{u} \psi$ (from 1, 12, by transitivity).

\end{enumerate}

\end{proof}

We define maximal consistent sets as before, but now we consider them closed over the principles of $\mathtt{HIU}$. The Lindenbaum's lemma can be proved standardly.

\begin{lemma}[Lindenbaum's lemma]
\label{LindIU}

Every consistent set of $\mathtt{HIU}$ can be extended to a maximally consistent one.

\end{lemma}

As before, let $\mathcal{X}^{C}$ be the set of all $\mathtt{HIU}$-maximally consistent sets. For each $\Gamma \in \mathcal{X}^{C}$, we define:

\begin{itemize}

\item $\Gamma[I^{u}] := \{ \phi \mid  \neg I^{u} \bar{\psi} \rightarrow \neg I^{u}(\phi \vee \psi) \in \Gamma$ for any $\psi \in \Gamma\}$.

\end{itemize}

The canonical model is now defined as follows.

\begin{defi}[Canonical Model for $\Gamma_{0}$]
\label{canmodIU}

Given a maximal consistent set $\Gamma_{0}$ of $\texttt{HIW}$, the canonical model for $\Gamma_{0}$ is a tuple $\mathcal{M}^{C} = \{W^{C}, R^{C}, v^{C}, \mathcal{T}^{C}\}$, where 

\begin{itemize}

\item $W^{C} = \{ \Gamma \in \mathcal{X}^{C} \mid \Gamma_{0} \sim \Gamma\}$;

\item $R^{C} \subseteq W^{C} \times W^{C}$ s.t. for all $\Gamma, \Delta \in W^{C}$, 

\begin{center}

$\Gamma R^{C} \Delta$ iff $\Gamma[I^{u}] \subseteq \Delta$;

\end{center}

\item $v^{C}(p) = \{\Gamma \in W^{C} \mid p \in \Gamma\}$;

\item $\mathcal{T}^{C}$ is such that

\begin{itemize}

\item $T^{C}= \{a, b\}$ where $a = \{p \in \mathtt{Prop} \mid I^{u} \bar{p} \in \Gamma_{0}\}$ and $b = \{ p \in \mathtt{Prop} \mid \neg I^{u} \bar{p} \in \Gamma_{0}\}$;

\item $\oplus^{C} : T^{C} \times T^{C} \mapsto T^{C}$ s.t. $a \oplus^{C} a = a$, $b \oplus^{C} b = b$, $a \oplus^{C} b = b \oplus^{C} a = a$;

\item $\mathfrak{K}^{C} = b$;

\item $t^{C}: \mathtt{Prop} \mapsto T^{C}$ s.t. for all $c \in T^{C}$ and $p \in \mathtt{Prop} : t^{C}(p) = c$ iff $p \in c$, and $t^{C}$ extends to the whole language by $t^{C}(\phi) = \oplus^{C}\{t^{C}(p) \mid p \in Var(\phi)\}$.

\end{itemize}

\end{itemize}

\end{defi}

\begin{prop}

For all $\phi \in \mathcal{L}_{I^{u}}$ and for any $\Gamma \in W^{C}: \neg I^{u} \bar{\phi} \in \Gamma$ iff $\forall p \in Var(\phi): \neg I^{u} \bar{p} \in \Gamma$. 

\end{prop}

\begin{proof}
\label{useful3}

The proof follows closely the proof of Lemma 9 \citep{Ozgun2023}.

\begin{itemize}

\item[$(\Rightarrow)$] Let $\neg I^{u} \bar{\phi} \in \Gamma$. Then, by (A8$_{I^{u}}$), $\neg I^{u} \overline{p} \in \Gamma$ for all $p \in Var(\phi)$.

\item[$(\Leftarrow)$] Let $Var(\phi) = \{p_{1}, ..., p_{n}\}$. Then, $\bar{\phi} = \bar{p_{1}} \wedge ... \wedge \bar{p_{n}}$. Let $\neg I^{u} \bar{p_{i}} \in \Gamma$ for all $p_{i} \in \{p_{1}, ..., p_{n}\}$. Then, $\bigwedge_{i \leq n} \neg I^{u} \bar{p_{i}} \in \Gamma$, because $\Gamma$ is a maximal consistent set. By (A3$_{I^{u}}$), $\neg I^{u} \bigwedge_{i \leq n} \bar{p_{i}} \in \Gamma$, that is $\neg I^{u} \bar{\phi} \in \Gamma$.

\end{itemize}

\end{proof}

\begin{prop}
\label{useful4}

Given a canonical topic-sensitive model $\mathcal{M}^{C} = \{W^{C}, R^{C}, v^{C}, \mathcal{T}^{C}\}$, for any $\Gamma \in W^{C}$, and $\phi \in \mathcal{L}_{I^{u}}$:

\begin{center}

$\neg I^{u} \bar{\phi} \in \Gamma$ iff $t^{C}(\phi) \sqsubseteq \mathfrak{K}^{C}$.

\end{center}

\end{prop}

\begin{proof}
The proof is the same as of Corollary 10 \citep{Ozgun2023}. We just replace $K$ by $\neg I^{u}$, Lemma 9 by Proposition \ref{useful3}, and Ax3$_{K}$ by (A4$_{I^{u}}$).

\end{proof}

\begin{lemma}[Truth Lemma]
\label{TLIU}

Let $\mathcal{M}^{C} = \{W^{C}, R^{C}, v^{C}, \mathcal{T}^{C}\}$ be the canonical model for $\Gamma_{0}$. Then, for all $\phi \in \mathcal{L}_{I^{u}}$ and $\Gamma \in W^{C}$, we have $\mathcal{M}^{C}, \Gamma \models \phi$ iff $\phi \in \Gamma$.

\end{lemma}

\begin{proof}

The proof is by induction on the complexity of $\phi$. The cases for propositional variables, boolean operators and $\Box$ are standard. The case of $\phi:= I^{u} \psi$ is as follows.

\begin{itemize}

\item[($\Rightarrow$)] Assume $\mathcal{M}^{C}, \Gamma \models I^{u} \psi$. This means that either $t^{C}(\psi) \not \sqsubseteq \mathfrak{K}^{C}$, or there exist $\Gamma'$ such that $\Gamma R^{C} \Gamma'$ and $\mathcal{M}^{C}, \Gamma' \models \neg \psi$. If $t^{C}(\psi) \not \sqsubseteq \mathfrak{K}^{C}$, then by Proposition \ref{useful4}, $\neg I^{u} \bar{\psi} \not \in \Gamma$, that is $I^{u} \bar{\psi} \in \Gamma$, and thus by (A5$_{I^{u}}$), $I^{u} \psi \in \Gamma$.

Let $t^{C}(\psi) \sqsubseteq \mathfrak{K}^{C}$. By induction hypothesis we have (i) $\psi \in \Gamma$ and (ii) $\neg \psi \in \Gamma'$. By definition of $\Gamma R^{C} \Gamma'$, this means that there exists $\chi \in \Gamma$, s.t. $\neg I^{u} \bar{\chi} \rightarrow \neg I^{u}(\psi \vee \chi) \not \in \Gamma$, and thus (iii) $\neg I^{u}\bar{\chi} \in \Gamma$ and (iv) $I^{u}(\psi \vee \chi) \in \Gamma$. From $t^{C}(\psi) \sqsubseteq \mathfrak{K}^{C}$, by Proposition \ref{useful4}, we have (v) $\neg I^{u} \bar{\psi} \in \Gamma$. From (i), (v), (iii), and (iv), by (A6$_{I^{u}}$), we have $I^{u} \psi \in \Gamma$.

\item[($\Leftarrow$)] Assume that $I^{u} \psi \in \Gamma$. If $t^{C}(\psi) \not \sqsubseteq \mathfrak{K}^{C}$, then, by the definition of $I^{u}$, we have $\mathcal{M}^{C}, \Gamma \models I^{u} \psi$. If $t^{C}(\psi) \sqsubseteq \mathfrak{K}^{C}$, then, by Proposition \ref{useful4}, we have that $\neg I^{u} \bar{\psi} \in \Gamma$. By (A1$_{I^{u}}$) we have $\psi \in \Gamma$. Then, we show that there is $\Gamma' \in W^{C}$ s.t. $\Gamma R^{C} \Gamma'$ and $\neg \psi \in \Gamma'$. This contributes to show that:

$\{\chi \mid \neg I^{u}\bar{\phi} \rightarrow \neg I^{u}(\chi \vee \phi) \in \Gamma$ for any $
\phi \in \Gamma \} \cup \{\neg \psi\}$ is consistent.

Assume that it is not consistent. Thus, (i) there exist $\chi_{1}, ..., \chi_{n}$ s.t. $\vdash (\chi_{1} \wedge ... \wedge \chi_{n}) \rightarrow \psi$ and (ii) $\neg I^{u} \bar{\phi} \rightarrow \neg I^{u}(\chi_{k} \vee \phi) \in \Gamma$ for all $k \in [1, n]$, for any $\phi \in \Gamma$. From $\neg I^{u} \bar{\psi} \in \Gamma$, $\psi \in \Gamma$ (from (A1$_{I^{u}}$)), and (ii), we have $\neg I^{u}(\chi_{k} \vee \psi) \in \Gamma$. By definition of the canonical model and Proposition \ref{useful4}, we have $\neg I^{u} \bar{\chi_{k}} \in \Gamma$. Then, by (NEC$_{I^{u}}$), we have $\neg I^{u} ((\chi_{1} \wedge ... \wedge \chi_{n}) \rightarrow \psi) \in \Gamma$. By Proposition \ref{genA7IU}, we have $\neg I^{u} \psi \in \Gamma$, which is a contradiction. Thus, by the definition of $R^{C}$, and Lemma \ref{LindIU}, we conclude that there exist $\Gamma'$ s.t. $\Gamma R^{C} \Gamma'$ and $\neg \psi \in \Gamma'$. 

By induction hypothesis and the definition of $I^{u}$, this means that $\mathcal{M}^{C}, \Gamma \models I^{u} \psi$.

\end{itemize}

\end{proof}

\begin{theor}

The system $\mathtt{HIU}$ is complete.

\end{theor}

\begin{proof}

Let $\not \vdash \phi$. Then $\{\neg \phi\}$ is consistent. By Lemma \ref{LindIU}, there exists a maximal consistent set $\Gamma_{0}$ s.t. $\phi \not \in \Gamma_{0}$. By Lemma \ref{TLIU}, we have $\mathcal{M}^{C}, \Gamma_{0} \not \models \phi$.

\end{proof}

\section{Proofs for system \texttt{HDI}}

\subsection{Soundness of \texttt{HDI}}
\label{appHDIsound}

\begin{theor}

The system $\mathtt{HDI}$ is sound.

\end{theor}

\begin{proof}

The cases of (CPL) and (S5$_{\Box}$) and standard.

The cases of (G1) - (G4) are straightforward from the definition of $G$ and the definition of topic model.

For the case of (G5), let $\mathcal{M}, w \models G \phi$ for an  arbitrary $w \in \mathcal{M}$. This means that $t(\phi) \sqsubseteq \mathfrak{K}$. Assume that $\mathcal{M}, w \not \models \Box G \phi$, that is there exists $w' \in \mathcal{M}$ s.t. $\mathcal{M}, w' \models \neg G \phi$, that is $t(\phi) \not \sqsubseteq \mathfrak{K}$, contradiction.

For the case of the rule (HIR), let (i) $\mathcal{M} \models \phi \rightarrow \psi$ for an arbitrary $\mathcal{M}$ (i.e., $\phi \rightarrow \psi$ is valid) and assume that (ii) $\mathcal{M}, w \models \phi$, (iii) $\mathcal{M}, w \models I \psi$, and  (iv) $\mathcal{M}, w \models G \phi$ for an arbitrary $w \in \mathcal{M}$. From (iv), we have (v) $t(\phi) \sqsubseteq \mathfrak{K}$. Let (vi) $\mathcal{M}, w \models \neg I^{d} \phi$. Then, either $t(\phi) \not \sqsubseteq \mathfrak{K}$ (which contradicts (v)), or $\mathcal{M}, w \models \neg \phi$ (which contradicts (ii)), or (viii) there exists $w'$ s.t. $w \not = w'$, $Rww'$ and $\mathcal{M}, w' \models \phi$. From (viii) and (i), we have that (ix) there exists $w'$ s.t. $w \not = w'$, $Rww'$ and $\mathcal{M}, w' \models \psi$ From (iii) we have that for all $w''$ s.t. $w\not = w''$, if $Rww''$ then $\mathcal{M}, w'' \models \neg \psi$, which contradicts to (ix). 

For the case of (A1$_{I^{d}}$), let $\mathcal{M}, w \models I^{d} \phi$ for an arbitrary $w \in \mathcal{M}$. By the definition of $I^{d}$ this means that $\mathcal{M}, w \models \phi$.

For the case of (A2$_{I^{d}}$), let $\mathcal{M}, w \models I^{d} \phi \wedge I^{d} \psi$ for an arbitrary $w \in \mathcal{M}$. This means that $t(\phi) \sqsubseteq \mathfrak{K}$, $t(\psi) \sqsubseteq \mathfrak{K}$ (and thus $t(\phi \vee \psi) \sqsubseteq \mathfrak{K}$), $\mathcal{M}, w \models \phi$, $\mathcal{M}, w \models \psi$ (and thus $\mathcal{M}, w \models \phi \vee \psi$), and for all $w'$ s.t. $w \not = w'$, if $Rww'$ then $\mathcal{M}, w' \models \phi$ and $\mathcal{M}, w'  \models \psi$ (and thus $\mathcal{M}, w' \models \phi \vee \psi$). Thus, $\mathcal{M}, w \models I^{d} (\phi \vee \psi)$.

For the case of (A3$_I^{d}$), let $\mathcal{M}, w \models I^{d} \phi$ for an arbitrary $\mathcal{M}$. Then, by definition of $I^{d}$, $t(\phi) \sqsubseteq \mathfrak{K}$, which means that $\mathcal{M}, w \models G \phi$.


\end{proof}

\subsection{Completeness of \texttt{HDI}}
\label{appHDIcomplete}

For the completeness proof, we need the following proposition which is a generalization of (A2$_{I^{d}}$). The proof of the proposition is straightforward from (A2$_{I^{d}}$).

\begin{prop}
\label{useful7}

For all $k \geq 1$:

\begin{center}

$\bigwedge^{k}_{j=1} I^{d} \phi_{j} \rightarrow I^{d} \bigvee^{k}_{j=1} \phi_{j}$

\end{center}

\end{prop}

We define maximal consistent sets as before, but now we consider them closed over the principles of $\mathtt{HDI}$. The Lindenbaum's lemma can be proved standardly.

\begin{lemma}[Lindenbaum's Lemma]
\label{LindID}

Every consistent set of $\mathtt{HDI}$ can be extended to a maximally consistent one.

\end{lemma}

As before, let $\mathcal{X}^{C}$ be the set of all $\mathtt{HDI}$-maximally consistent sets. For each $\Gamma \in \mathcal{X}^{C}$, we define:

\begin{itemize}

\item $\Gamma[I^{d}] := \{ \neg \phi \mid I^{d} \phi \in \Gamma\}$.

\end{itemize}

The definition of $\Gamma[I^{d}]$ is the one proposed by \citet{Gilbert2021}.

We will also need the following proposition.

\begin{prop}
\label{useful8}

Let $\Gamma$ be a $\mathtt{HDI}$ maximally consistent set. If $I^{d} \phi \not \in \Gamma$,  $G\phi \in \Gamma$, and $\phi \in \Gamma$, then $\{\phi\} \cap \Gamma[I^{d}]$ is consistent. If $\Gamma[I^{d}]$ is non-empty, then $\{\phi\} \cap \Gamma[I^{d}] \subseteq \Gamma'$ for some $\mathtt{HDI}$-maximally consistent set $\Gamma'$ such that $\Gamma \not = \Gamma'$.

\end{prop}

\begin{proof}

Assume that $\{\phi\} \cap \Gamma[I^{d}]$ is inconsistent. Then, $\Gamma[I^{d}]$ is non-empty. Otherwise $\{\phi\} \cap \Gamma[I^{d}] = \{\phi\}$, and since $\phi \in \Gamma$, $\{\phi\} \cap \Gamma[I^{d}]$ would be consistent. Then, we have $\vdash \neg (\phi \wedge \neg \psi_{1} \wedge ... \wedge \neg \psi_{n})$ where $\neg \psi_{k} \in \Gamma[I^{d}]$ for $k \in [1,n]$. Then, $\vdash \phi \rightarrow (\psi_{1} \vee ... \vee \psi_{n})$. From (HIR), $\vdash \phi \rightarrow (I^{d}(( \psi_{1} \vee ... \vee \psi_{n}) \rightarrow (G \phi \rightarrow I^{d} \phi)))$, and thus $I^{d}(( \psi_{1} \vee ... \vee \psi_{n}) \rightarrow (G \phi \rightarrow I^{d} \phi)) \in \Gamma$ (because $\phi \in \Gamma$). From the fact that $I^{d} \psi_{i} \in \Gamma$ for each $\psi_{i}$, we have $I^{d} \psi_{1} \wedge ... \wedge I^{d} \psi_{n} \in \Gamma$. By Proposition \ref{useful7}, $I^{d}(\psi_{1} \vee ... \vee \psi_{n}) \in \Gamma$. By having $G\phi \in \Gamma$, we obtain $I^{d} \phi \in \Gamma$, which contradicts the consistency of $\Gamma$.

Since $\{\phi\} \cap \Gamma[I^{d}]$ is consistent, it is contained in some $\mathtt{HDI}$ maximally consistent set $\Gamma'$. Let $\Gamma[I^{d}]$ be non-empty. Then, there is a formula $\psi$ s.t. $\neg \psi \in \Gamma[I^{d}]$ and thus $I^{d} \psi \in \Gamma$. By (A1$_{I^{d}}$), $\psi \in \Gamma$ while $\neg \psi \in \Gamma'$. Thus, $\Gamma \not = \Gamma'$.

\end{proof}

The canonical model is now defined as follows.

\begin{defi}[Canonical Model for $\Gamma_{0}$]
\label{canmodID}

Given a maximal consistent set $\Gamma_{0}$ of $\texttt{HIW}$, the canonical model for $\Gamma_{0}$ is a tuple $\mathcal{M}^{C} = \{W^{C}, R^{C}, v^{C}, \mathcal{T}^{C}\}$, where 

\begin{itemize}

\item $W^{C} = \{ \Gamma \in \mathcal{X}^{C} \mid \Gamma_{0} \sim \Gamma\}$;

\item $R^{C} \subseteq W^{C} \times W^{C}$ s.t. for all $\Gamma, \Delta \in W^{C}$, 

\begin{center}

$\Gamma R^{C} \Delta$ iff $\Gamma[I^{d}] \subseteq \Delta$;

\end{center}

\item $v^{C}(p) = \{\Gamma \in W^{C} \mid p \in \Gamma\}$;

\item $\mathcal{T}^{C}$ is such that

\begin{itemize}

\item $T^{C}= \{a, b\}$ where $a = \{p \in \mathtt{Prop} \mid G p \in \Gamma_{0}\}$ and $b = \{ p \in \mathtt{Prop} \mid \neg G p \in \Gamma_{0}\}$;

\item $\oplus^{C} : T^{C} \times T^{C} \mapsto T^{C}$ s.t. $a \oplus^{C} a = a$, $b \oplus^{C} b = b$, $a \oplus^{C} b = b \oplus^{C} a = a$;

\item $\mathfrak{K}^{C} = b$;

\item $t^{C}: \mathtt{Prop} \mapsto T^{C}$ s.t. for all $c \in T^{C}$ and $p \in \mathtt{Prop} : t^{C}(p) = c$ iff $p \in c$, and $t^{C}$ extends to the whole language by $t^{C}(\phi) = \oplus^{C}\{t^{C}(p) \mid p \in Var(\phi)\}$.

\end{itemize}

\end{itemize}

\end{defi}

\begin{prop}
\label{useful5}

For all $\phi \in \mathcal{L}_{I^{d}, G}$ and for any $\Gamma \in W^{C}: G \phi \in \Gamma$ iff $\forall p \in Var(\phi): G p \in \Gamma$.

\end{prop}

\begin{proof}

\begin{itemize}

\item[$(\Rightarrow)$] Let $G \phi \in \Gamma$. By (G1) -- (G4), we have $Gp \in \Gamma$ for all $p \in Var(\phi)$.

\item[$(\Leftarrow)$] Let $G p \in \Gamma$ for all $p \in Var(\phi)$. Consider formula $\phi \vee \neg \phi$, that is $\bigwedge_{p \in Var(\phi)} (p \vee \neg p)$. By (G1), (G4), and our assumption, $G \bigwedge_{p \in Var(\phi)} (p \vee \neg p) \in \Gamma$, that is $G (\phi \vee \neg \phi) \in \Gamma$, that is, by (G1) and (G4), $G \phi \in \Gamma$.

\end{itemize}

\end{proof}

\begin{prop}
\label{useful6}

Given a canonical topic-sensitive model $\mathcal{M}^{C} = \{W^{C}, R^{C}, v^{C}, \mathcal{T}^{C}\}$, for any $\Gamma \in W^{C}$, and $\phi \in \mathcal{L}_{I^{d}, G}$:

\begin{center}

$G \phi$ iff $t^{C}(\phi) \sqsubseteq \mathfrak{K}^{C}$.

\end{center}

\end{prop}

\begin{proof}

The proof is the same as of Corollary 10 \citep{Ozgun2023} We just replace $K\bar{\phi}$ with $G\phi$, Lemma 9 by Lemma \ref{useful5}, and Ax3$_{K}$ by (G5).

\end{proof}

\begin{lemma}[Truth Lemma]
\label{TLHDI}

Let $\mathcal{M}^{C} = \{W^{C}, R^{C}, v^{C}, \mathcal{T}^{C}\}$ be the canonical model for $\Gamma_{0}$. Then, for all $\phi \in \mathcal{L}_{I^{d}, G}$ and $\Gamma \in W^{C}$, we have $\mathcal{M}^{C}, \Gamma \models \phi$ iff $\phi \in \Gamma$.

\end{lemma}

\begin{proof}

The proof is by induction on the complexity of $\phi$. The cases for propositional variables, boolean operators and $\Box$ are standard. The case of $G$ follows straightforwardly from Proposition \ref{useful6} and definition of $G$. The case of $\phi : = I^{d} \psi$ is as follows.

\begin{itemize}

\item[($\Rightarrow$)] Assume $I^{d} \psi \not \in \Gamma$. Clearly, either $t^{C}(\psi) \sqsubseteq \mathfrak{K}^{C}$, or $t(\psi) \not \sqsubseteq \mathfrak{K}^{C}$. In the second case, we get $\mathcal{M}, \Gamma \not \models I^{d} \psi$ by the definition of $I^{d}$. Thus, assume that $t^{C}(\psi) \sqsubseteq \mathfrak{K}^{C}$, i.e., $G \psi \in \Gamma$. Since $\Gamma$ is maximal, either $\psi$ or $\neg \psi \in \Gamma$. In the second case, $\psi \not \in \Gamma$ and thus, by induction hypothesis, $\mathcal{M}^{C}, \Gamma \not \models \psi$. Then, $\mathcal{M}^{C}, \Gamma \not \models I^{d} \psi$.

Let $\psi \in \Gamma$, which means that $\mathcal{M}^{C}, \Gamma \models \psi$. From Proposition \ref{useful8}, $\{\psi\} \cap \Gamma[I^{d}]$ is consistent and contained in some $\Gamma'$. Thus, by definition of $R^{C}$, $\Gamma R^{C} \Gamma'$ and $\mathcal{M}^{C}, \Gamma' \models \psi$. When $\Gamma[I^{d}] \not = \emptyset$, $\Gamma \not = \Gamma'$, and thus $\mathcal{M}^{C}, \Gamma \not \models I^{d} \psi$.

If $\Gamma[I^{d}] = \emptyset$, then consider any $p$ which does not occur in $\psi$. Either $p \in \Gamma$ or $\neg p \in \Gamma$. Having in mind that $p$ does not occur in $\psi$, we have that $\{\psi \wedge p\}$ and $\{\psi \wedge \neg p\}$ are consistent. Then, whether it is $p \in \Gamma$ or $\neg p \in \Gamma$, we have that there is some $\Gamma'$ such that $\psi \in \Gamma'$ and $\Gamma' \not = \Gamma$. By induction hypothesis $\mathcal{M}^{C}, \Gamma' \models \psi$. Since $\Gamma[I^{d}] = \emptyset$, $\Gamma[I^{d}] \subseteq \Gamma'$, and $\Gamma R^{C} \Gamma'$. Thus, $\mathcal{M}^{C}, \Gamma \models \psi$ but $\mathcal{M}^{C}, \Gamma' \models \psi$ for some $\Gamma' \not = \Gamma$ and $\Gamma R^{C} \Gamma'$. Thus, $\mathcal{M}^{C}, \Gamma \not \models I^{d} \psi$.

\item[($\Leftarrow$)] Assume $I^{d} \psi \in \Gamma$. From (A3$_{I^{d}}$), we have $G \psi \in \Gamma$, that is $t^{C}(\psi) \sqsubseteq \mathfrak{K}^{C}$. By (A1$_{I^{d}}$), $\psi \in \Gamma$ and thus $\mathcal{M}^{C}, \Gamma \models \psi$. Consider and arbitrary $\Gamma' \in W$ such that $\Gamma R^{C} \Gamma'$. Then, $\Gamma[I^{d}] \subseteq \Gamma'$, and thus $\neg \psi \in \Gamma'$. Since $\Gamma' \in W$, and thus it is consistent, $\psi \not \in \Gamma'$, and thus $\mathcal{M}^{C}, \Gamma' \not \models \psi$. Thus, $\mathcal{M}^{C}, \Gamma \models I^{d} \psi$.

\end{itemize}

\end{proof}

\begin{theor}

The system $\mathtt{HDI}$ is complete.

\end{theor}

\begin{proof}

Let $\not \vdash \phi$. Then $\{\neg \psi\}$ is consistent. By Lemma \ref{LindID}, there exists a maximal consistent set $\Gamma_{0}$ s.t. $\phi \not \in \Gamma_{0}$. By Lemma \ref{TLHDI}, we have $\mathcal{M}^{C}, \Gamma_{0} \not \models \phi$.

\end{proof}

\end{document}